\nonstopmode
\documentclass{amsart}
\usepackage{amsmath, amssymb,amsthm}
\usepackage{hyperref}
\usepackage{cleveref}

\theoremstyle{plain}
\newtheorem*{theorem*}{Theorem}
\newtheorem{theorem}{Theorem}[section]
\newtheorem*{lemma*}{Lemma}
\newtheorem{lemma}[theorem]{Lemma}
\newtheorem*{proposition*}{Proposition}
\newtheorem{proposition}[theorem]{Proposition}
\newtheorem*{corollary*}{Corollary}
\newtheorem{corollary}[theorem]{Corollary}
\newtheorem*{claim*}{Claim}

\theoremstyle{definition}
\newtheorem*{definition*}{Definition}

\newtheorem*{example*}{Example}
\newtheorem{example}[theorem]{Example}

\newtheorem*{remark*}{Remark}
\newtheorem*{remarks*}{Remarks}
\newtheorem{remark}[theorem]{Remark}

\def\al{\alpha}
\def\be{\beta}

\def\de{\delta}

\def\ve{\varepsilon}

\def\la{\lambda}

\def\si{\sigma}

\def\om{\omega}

\def\C{\mathbb{C}}
\def\CC{\mathbb{C}}

\def\N{\mathbb{N}}
\def\NN{\mathbb{N}}

\def\R{\mathbb{R}}
\def\RR{\mathbb{R}}
\def\Z{\mathbb{Z}}

\def\cC{\mathcal{C}}

\def\cE{\mathcal{E}}

\def\cH{\mathcal{H}}

\def\cL{\mathcal{L}}
\def\cM{\mathcal{M}}
\def\cN{\mathcal{N}}

\def\cR{\mathcal{R}}

\def\cV{\mathcal{V}}

\def\fF{\mathfrak{F}}
\def\fG{\mathfrak{G}}

\def\fc{\mathfrak{c}}

\def\<{\langle}
\def\>{\rangle}

\renewcommand{\preceq}{\preccurlyeq}
\let\on=\operatorname

\def\RR{\mathbb R}
\def\NN{\mathbb N}

\def\A{\;\forall}
\def\E{\;\exists}
\def\oo{\infty}

\title{Ultradifferentiable classes of entire functions}
\author{David Nicolas Nenning and Gerhard Schindl}

\address{D. N.~Nenning: Fakult\"at f\"ur Mathematik, Universit\"at Wien, Oskar-Morgenstern-Platz~1, A-1090 Wien, Austria.}
\email{david.nicolas.nenning@univie.ac.at}

\address{G.~Schindl: Fakult\"at f\"ur Mathematik, Universit\"at Wien, Oskar-Morgenstern-Platz~1, A-1090 Wien, Austria.}
\email{gerhard.schindl@univie.ac.at}

\begin{document}
\begin{abstract}
We study classes of ultradifferentiable functions defined in terms of small weight sequences violating standard growth and regularity requirements. First, we show that such classes can be viewed as weighted spaces of entire functions for which the crucial weight is given by the associated weight function of the so-called conjugate weight sequence. Moreover, we generalize results from M. Markin from the so-called small Gevrey-setting to arbitrary convenient families of (small) sequences and show how the corresponding ultradifferentiable function classes can be used to detect boundedness of normal linear operators on Hilbert spaces (associated to an evolution equation problem). Finally, we study the connection between small sequences and the recent notion of dual sequences introduced in the PhD-thesis of J. Jim\'{e}nez-Garrido.
\end{abstract}

\thanks{D. N. Nenning and G. Schindl are supported by Austrian Science Fund (FWF)-Project P33417-N}
\keywords{Weight sequences, associated weight functions, growth and regularity properties for sequences, weighted spaces of entire functions, boundedness of linear operators}
\subjclass[2020]{26A12, 30D15, 34G10, 46A13, 46E10, 47B02}
\date{\today}

\maketitle

\section{Introduction}
Spaces of ultradifferentiable functions are sub-classes of smooth functions with certain restrictions on the growth of their derivatives. Two classical approaches are commonly considered, either the restrictions are expressed by means of a weight sequence $M=(M_p)_{p\in\NN}$, also called {\itshape Denjoy-Carleman classes} (e.g. see \cite{Komatsu73}), or by means of a weight function $\omega$ also called {\itshape Braun-Meise-Taylor classes}, see \cite{BraunMeiseTaylor90}. In this work we are exclusively dealing with the weight sequence approach.

More precisely (in the one-dimensional case) for each compact set $K$, the set
	\begin{equation}\label{introequ}
		\left\{\frac{f^{(p)}(x)}{h^pM_p}\,\,:\,\, p\in\NN,\,\, x\in K \right\}
	\end{equation}
	is required to be bounded. Naturally, one can consider two different types of spaces: For the {\itshape Roumieu-type} the boundedness of the set in \eqref{introequ} is required for {\itshape some} $h>0$, whereas for the {\itshape Beurling-type} it is required for {\itshape all} $h>0$.\vspace{6pt}

 In the literature standard growth and regularity conditions are assumed for $M$, roughly speaking one is interested in sufficiently fast growing sequences $M$ in order to ensure that $M_p$ is (much) larger than $p!$ for all $p\in\NN$. This is related to the fact that for such sequences the corresponding function spaces are lying between the real-analytic functions and the class of smooth functions. So classes being (strictly) smaller than the spaces corresponding to the sequence $(p!)_{p\in\NN}$ are excluded due to these basic requirements. Moreover, the regularity condition {\itshape log-convexity}, i.e. $(M.1)$ in \cite{Komatsu73}, is more or less standard and even $M\in\hyperlink{LCset}{\mathcal{LC}}$ is basic, see Section \ref{definitionweightsequsect} for the definition of this set. (Formally, if log-convexity for $M$ fails, then one might avoid technical complications by passing to its so-called log-convex minorant.) The analogous notion of {\itshape log-concavity} has not been used in the ultradifferentiable setting.\vspace{6pt}

 The (most) well-known examples are the so-called {\itshape Gevrey sequences} of type $\alpha>0$ with $G^{\alpha}_p:= p!^{\alpha}$ (or equivalently use $M^{\alpha}_p:=p^{p\alpha}$) and this one-parameter family illustrates this behavior when considering different values of the crucial parameter $\alpha$: Usually in the literature one is interested in $\alpha>1$ and the limiting case $\alpha=1$ for the Roumieu-type precisely yields the real-analytic functions. Indices $0<\alpha<1$ give a non-standard setting and the corresponding function classes are tiny (''small Gevrey setting''). At this point let us make aware that we are using for the sequence $M$ the notation ''including the factorial term'' in \eqref{introequ} since in the literature occasionally authors also deal with $\frac{f^{(p)}(x)}{h^pp!M_p}$, e.g. in \cite{Thilliezdivision}, and so $M$ in these works corresponds to the sequence $m$ in the notation used in this paper (see Example \ref{Gevreyexample}). On the other hand the crucial conditions on the sequences appearing in this work illustrate the relevance of the difference between $m$ and $M$, see the assumptions in Section \ref{answersection}.\vspace{6pt}

However, from an abstract mathematical point of view it is interesting and makes sense to study also ultradifferentiable classes defined by non-standard/small sequences and to ask the following questions:

\begin{itemize}
\item[$(i)$] What are the differences between such small classes and spaces defined in terms of ``standard sequences''?

\item[$(ii)$] What is the importance of such small spaces and for which applications can they be useful?

\item[$(iii)$] Can we transfer known results from the standard setting, e.g. the characterization of inclusion relations for function spaces in terms of the corresponding weight sequences, to small spaces?

\item[$(iv)$] Does there exist a close resp. canonical relation between standard and non-standard sequences, or more precisely: Can one construct from a given standard sequence a small one (and vice versa)?
\end{itemize}

The aim of this article is to focus on these problems. Indeed, question $(iv)$ has served as the main motivation and the starting point for writing this work. Very recently, in \cite{dissertationjimenez} we have introduced the notion of the {\itshape dual sequence.} For each given standard $M$, e.g. if  $M\in\hyperlink{LCset}{\mathcal{LC}}$, it is possible to introduce the dual sequence $D$, see Section \ref{dualsection} for precise definitions and citations. In \cite{dissertationjimenez} this notion and the relation between $M$ and $D$ has been exclusively studied by considering growth and regularity indices (which are becoming relevant in the so-called ultraholomorphic setting). The aim is now to study further applications of this new notion and the conjecture is that for ``nice large standard sequences'' $M$ the corresponding dual sequence $D$ is a ``convenient small one'' which allows to study a non-standard setting.\vspace{6pt}

The literature concerning small ultradifferentiable function classes is non-exhaustive and to the best of our knowledge we have only found works by M. Markin treating the small Gevrey setting, see \cite{MarkinI}, \cite{MarkinII}, and \cite{MarkinIII}. More precisely, the goal there has been: Given a Hilbert space $H$ and a normal (unbounded) operator $A$ on $H$, then consider the associated evolution equation
$$y'(t) = A y(t),$$
and one asks the following question: Is a priori known smoothness of all (weak) solutions of this equation sufficient to get that the operator $A$ is bounded? Markin has studied this problem within the small Gevrey setting, i.e. it has been shown that if each weak solution of this evolution equation belongs to some {\itshape small Gevrey class,} then the operator $A$ is bounded. In order to proceed Markin considers (small) Gevrey classes with values in a Hilbert space. Based on this knowledge one can then study if for different small classes Markin's results also apply and if one can generalize resp. strengthen his approach.\vspace{6pt}

The paper is structured as follows: In Section \ref{definitionsection} we introduce the notion of the so-called conjugate sequence $M^{*}$ (see \eqref{Mstardef}), we collect and compare all relevant (non-)standard growth and regularity assumptions on $M$ and $M^{*}$ and define the corresponding function classes.

In Section \ref{ultradiffvsweightedentiresect} we treat question $(i)$ and show that classes defined by small sequences $M$ are isomorphic (as locally convex vector spaces) to weighted spaces of entire functions, see the main result Theorem \ref{thm:main}. Thus we are generalizing the auxiliary result \cite[Lemma 3.1]{MarkinIII} from the small Gevrey setting, see Section \ref{sec:Markinascor} for the comparison. The crucial weight in the weighted entire setting is given in terms of the so-called associated weight $\omega_{M^{*}}$ (see Section \ref{assoweightfctset}) and so expressed in terms of the conjugate sequence $M^{*}$.

As an application of this statement, concerning problem $(iii)$ above, we characterize for such small classes the inclusion relations in terms of the defining (small) sequences, see Theorem \ref{weightholombysequcharact}. This is possible by combining Theorem \ref{thm:main} with recent results for the weighted entire setting obtained by the second author in \cite{inclusion}.

Section \ref{boundednesssection} is dedicated to problem $(ii)$ and the study resp. the generalization of Markin's results. We introduce more general families of appropriate small sequences and extend the sufficiency testing criterion for the boundedness of the operator $A$ to these sets.

Finally, in the Appendix \ref{dualsection} we focus on $(iv)$ and show that dual sequences are serving as examples for non-standard sequences and hence this framework is establishing a close relation between known examples for weight sequences in the literature and small sequences for which the main results in this work can be applied (see Theorem \ref{dual14} and Corollary \ref{dual14cor}).\vspace{6pt}

\textbf{Acknowledgements.} The authors thank the anonymous referee for carefully reading this article and the valuable suggestions which have improved and clarified the presentation of the results. In particular, for the comments which we have summarized in $(ii)$ in Remark \ref{mainthmremark} and in Section \ref{Matuszeskasection}.

\section{Definitions and notations}\label{definitionsection}
\subsection{Basic notation}
We write $\NN:=\{0,1,2,\dots\}$ and $\NN_{>0}:=\{1,2,\dots\}$. Given a multi-index $\alpha=(\alpha_1,\dots,\alpha_d)\in\NN^d$ we set $|\alpha|:=\alpha_1+\dots+\alpha_d$. With $\mathcal{E}$ we denote the class of all smooth functions and with $\cH(\CC)$ the class of entire functions.

\subsection{Weight sequences}\label{definitionweightsequsect}
Let $M=(M_p)_p\in\RR_{>0}^{\NN}$, we introduce also $m=(m_p)_p$ defined by $m_p:=\frac{M_p}{p!}$ and $\mu=(\mu_p)_p$ by $\mu_p:=\frac{M_p}{M_{p-1}}$, $p\ge 1$, $\mu_0:=1$. $M$ is called {\itshape normalized} if $1=M_0\le M_1$ holds true. If $M_0=1$, then $M_p=\prod_{i=1}^p\mu_i$ for all $p\in\NN$.

$M$ is called {\itshape log-convex}, denoted by \hypertarget{lc}{$(\text{lc})$} and abbreviated by $(M.1)$ in \cite{Komatsu73}, if
$$\forall\;p\in\NN_{>0}:\;M_p^2\le M_{p-1} M_{p+1}.$$
This is equivalent to the fact that $\mu$ is non-decreasing. If $M$ is log-convex and normalized, then both $M$ and $p\mapsto(M_p)^{1/p}$ are non-decreasing. In this case we get $M_p\ge 1$ for all $p\ge 0$ and
\begin{equation}\label{mucompare1}
\forall\;p\in\NN_{>0}:\;\;\;(M_p)^{1/p}\le\mu_p.
\end{equation}
Moreover, $M_pM_q\le M_{p+q}$ for all $p,q\in\NN$.

In addition, for $M=(M_p)_p\in\RR_{>0}^{\NN}$ it is known that
\begin{equation}\label{mucompare}
\liminf_{p\rightarrow+\infty}\mu_p\le\liminf_{p\rightarrow+\infty}(M_p)^{1/p}\le\limsup_{p\rightarrow+\infty}(M_p)^{1/p}\le\limsup_{p\rightarrow+\infty}\mu_p.
\end{equation}

For convenience we introduce the following set of sequences:
$$\hypertarget{LCset}{\mathcal{LC}}:=\{M\in\RR_{>0}^{\NN}:\;M\;\text{is normalized, log-convex},\;\lim_{p\rightarrow+\infty}(M_p)^{1/p}=+\infty\}.$$
We see that $M\in\hyperlink{LCset}{\mathcal{LC}}$ if and only if $1=\mu_0\le\mu_1\le\dots$ with $\lim_{p\rightarrow+\infty}\mu_p=+\infty$ (see e.g. \cite[p. 104]{compositionpaper}) and there is a one-to-one correspondence between $M$ and $\mu=(\mu_p)_p$ by taking $M_p:=\prod_{i=0}^p\mu_i$.\vspace{6pt}

$M$ has {\itshape moderate growth}, denoted by \hypertarget{mg}{$(\text{mg})$}, if
$$\exists\;C\ge 1\;\forall\;p,q\in\NN:\;M_{p+q}\le C^{p+q+1} M_p M_q.$$
A weaker condition is {\itshape derivation closedness}, denoted by \hypertarget{dc}{$(\text{dc})$}, if
$$\exists\;A\ge 1\;\forall\;p\in\NN:\;M_{p+1}\le A^{p+1} M_p\Leftrightarrow\mu_{p+1}\le A^{p+1}.$$
It is immediate that both conditions are preserved under the transformation $(M_p)_p\mapsto(M_pp!^s)_p$, $s\in\RR$ arbitrary. In the literature (mg) is also known under {\itshape stability of ultradifferential operators} or $(M.2)$ and (dc) under $(M.2)'$, see \cite{Komatsu73}.\vspace{6pt}

$M$ has \hypertarget{beta1}{$(\beta_1)$} (named after \cite{petzsche}) if
$$\exists\;Q\in\NN_{>0}:\;\liminf_{p\rightarrow+\infty}\frac{\mu_{Qp}}{\mu_p}>Q,$$
and \hypertarget{gamma1}{$(\gamma_1)$} if
$$\sup_{p\in\NN_{>0}}\frac{\mu_p}{p}\sum_{k\ge p}\frac{1}{\mu_k}<+\infty.$$
In \cite[Proposition 1.1]{petzsche} it has been shown that for $M\in\hyperlink{LCset}{\mathcal{LC}}$ both conditions are equivalent and in the literature \hyperlink{gamma1}{$(\gamma_1)$} is also called ''strong non-quasianalyticity condition''. In \cite{Komatsu73} this is denoted by $(M.3)$. (In fact, there $\frac{\mu_p}{p}$ is replaced by $\frac{\mu_p}{p-1}$ for $p\ge 2$ but which is equivalent to having \hyperlink{gamma1}{$(\gamma_1)$}.)

A weaker condition on $M$ is \hypertarget{beta3}{$(\beta_3)$} (named after \cite{dissertation}, see also \cite{BonetMeiseMelikhov07}) which reads as follows:
$$\exists\;Q\in\NN_{>0}:\;\liminf_{p\rightarrow+\infty}\frac{\mu_{Qp}}{\mu_p}>1.$$

For two weight sequences $M=(M_p)_{p\in\NN}$ and $N=(N_p)_{p\in\NN}$ we write $M\le N$ if $M_p\le N_p$ for all $p\in\NN$ and $M\hypertarget{mpreceq}{\preceq}N$ if
$$\sup_{p\in\NN_{>0}}\left(\frac{M_p}{N_p}\right)^{1/p}<+\infty.$$
$M$ and $N$ are called equivalent, denoted by $M\hypertarget{approx}{\approx}N$, if
$$M\hyperlink{mpreceq}{\preceq}N\;\text{and}\;N\hyperlink{mpreceq}{\preceq}M.$$
Finally, we write $M\hypertarget{triangle}{\vartriangleleft}N$, if
$$\lim_{p\rightarrow+\infty}\left(\frac{M_p}{N_p}\right)^{1/p}=0.$$
In the relations above one can replace $M$ and $N$ simultaneously by $m$ and $n$ because $M\hyperlink{mpreceq}{\preceq}N\Leftrightarrow m\hyperlink{mpreceq}{\preceq}n$ and $M\hyperlink{triangle}{\vartriangleleft}N\Leftrightarrow m\hyperlink{triangle}{\vartriangleleft}n$.

For any $\alpha\ge 0$ we set
\[
G^{\alpha}:=(p!^{\alpha})_{p\in\NN}.
\]
So for $\alpha>0$ this denotes the classical {\itshape Gevrey sequence} of index/order $\alpha$.

\subsection{Classes of ultradifferentiable functions}\label{ultradiffclasssection}
	Let $M\in\RR_{>0}^{\NN}$, $U\subseteq\RR^d$ be non-empty open and for $K\subseteq\RR^d$ compact we write $K\subset\subset U$ if $\overline{K}\subseteq U$, i.e., $K$ is in $U$ relatively compact. We introduce now the following spaces of ultradifferentiable function classes.
	First, we define the (local) classes of {\itshape Roumieu-type} by
	$$\mathcal{E}_{\{M\}}(U):=\{f\in\mathcal{E}(U):\;\forall\;K\subset\subset U\;\exists\;h>0:\;\|f\|_{M,K,h}<+\infty\},$$
	and the classes of {\itshape Beurling-type} by
	$$\mathcal{E}_{(M)}(U):=\{f\in\mathcal{E}(U):\;\forall\;K\subset\subset U\;\forall\;h>0:\;\|f\|_{M,K,h}<+\infty\},$$
	where we denote
	\begin{equation*}\label{semi-norm-2}
		\|f\|_{M,K,h}:=\sup_{\alpha\in\NN^d,x\in K}\frac{|f^{(\alpha)}(x)|}{h^{|\alpha|} M_{|\alpha|}}.
	\end{equation*}
	
	For a sufficiently regular compact set $K$ (e.g. with smooth boundary and such that $\overline{K^\circ}=K$)
	$$\mathcal{E}_{M,h}(K):=\{f\in\mathcal{E}(K): \|f\|_{M,K,h}<+\infty\}$$
	is a Banach space and so we have the following topological vector spaces
	$$\mathcal{E}_{\{M\}}(K):=\underset{h>0}{\varinjlim}\;\mathcal{E}_{M,h}(K),$$
	and
	\begin{equation*}
		\mathcal{E}_{\{M\}}(U)=\underset{K\subset\subset U}{\varprojlim}\;\underset{h>0}{\varinjlim}\;\mathcal{E}_{M,h}(K)=\underset{K\subset\subset U}{\varprojlim}\;\mathcal{E}_{\{M\}}(K).
	\end{equation*}
	Similarly, we get
	$$\mathcal{E}_{(M)}(K):=\underset{h>0}{\varprojlim}\;\mathcal{E}_{M,h}(K),$$
	and
	\begin{equation*}
		\mathcal{E}_{(M)}(U)=\underset{K\subset\subset U}{\varprojlim}\;\underset{h>0}{\varprojlim}\;\mathcal{E}_{M,h}(K)=\underset{K\subset\subset U}{\varprojlim}\;\mathcal{E}_{(M)}(K).
	\end{equation*}

The spaces $\mathcal{E}_{\{M\}}(U)$ and $\mathcal{E}_{(M)}(U)$ are endowed with their natural topologies w.r.t. the above representations. We write $\mathcal{E}_{[M]}$ if we mean either $\mathcal{E}_{\{M\}}$ or $\mathcal{E}_{(M)}$ but not mixing the cases. We omit writing the open set $U$ if we do not want to specify the set where the functions are defined and formulate statements on the level of classes.\par
Usually one only considers real or complex valued functions, but we can analogously also define classes with values in Hilbert or even Banach spaces (for simplicity we assume in this case that the domain $U$ is contained in $\R$) by simply using
\[
\|f\|_{M,K,h}:= \sup_{p \in \N, x\in K}\frac{\|f^{(p)}(x)\|}{h^p M_p},
\]
in the respective definition, i.e. only the absolute value of $f^{(p)}(x)$ is replaced by the norm in the Banach space. Observe that the (complex) derivative of a function with values in a Banach space is defined in complete analogy to the complex valued case.
If we want to emphasize that the codomain is a Hilbert (or Banach) space $H$, we write $\cE_{[M]}(U,H)$. In analogy to that also $\cE(U,H)$ shall denote the $H$-valued smooth functions on $U$.

\begin{remark}
    \label{rem:inclusion}
Let $M,N\in\RR_{>0}^{\NN}$, the following is well-known, see e.g. \cite[Prop. 2.12]{compositionpaper}:

\begin{itemize}
\item[$(*)$] The relation $M\hyperlink{triangle}{\vartriangleleft}N$ implies $\mathcal{E}_{\{M\}}\subseteq\mathcal{E}_{(N)}$ with continuous inclusion. Similarly, $M\hyperlink{mpreceq}{\preceq}N$ implies $\mathcal{E}_{[M]}\subseteq\mathcal{E}_{[N]}$  with continuous inclusion.

\item[$(*)$] If $M\in\RR_{>0}^{\NN}$ is log-convex (and normalized) and $\mathcal{E}_{\{M\}}(\RR)\subseteq\mathcal{E}_{(N)}(\RR)$ (as sets) then by the existence of so-called $M$-characteristic functions, see \cite[Lemma 2.9]{compositionpaper}, \cite[Thm. 1]{thilliez} and the proof in \cite[Prop. 3.1.2]{diploma}, we get $M\hyperlink{triangle}{\vartriangleleft}N$ as well.
\end{itemize}
\end{remark}

\subsection{Ultradifferentiable classes of entire functions}
\label{sec:ultraasentire}

We shall tacitly assume that a holomorphic function on (an open subset of) $\C$ may have values in a Hilbert or even Banach space. The main theorems of one variable complex analysis (Cauchy integral formula, power series representation of holomorphic functions,...) hold mutatis mutandis, by virtue of the Hahn-Banach theorem, just as in the complex valued case.

First let us recall that for any open (and connected) set $U \subseteq \R$ the space $\mathcal{E}_{(G^1)}(U, H)$ can be identified with  $\cH(\C, H)$, the class of entire functions and both spaces are isomorphic as Fr\'echet spaces. The isomorphism $\cong$ is given by
\begin{equation*}
    \label{eq:extension}
    E:\cE_{(G^1)}(U,H) \rightarrow \cH(\C,H),\quad f \mapsto E(f):=\sum_{k=0}^{+\oo}\frac{f^{(k)}(x_0)}{k!}z^k,
\end{equation*}
where $x_0$ is any fixed point in $U$. The inverse is given by restriction to $U$, and its continuity follows easily from the Cauchy inequalities.

We apply the observation from Remark \ref{rem:inclusion}  to $N\equiv G^1$.

\begin{lemma}\label{smallclasseslemma}
Let $M\in\RR_{>0}^{\NN}$ be given.
\begin{itemize}
\item[$(i)$] If $\lim_{p\rightarrow+\infty}(m_p)^{1/p}=0$, then $\mathcal{E}_{\{M\}}\subseteq\mathcal{E}_{(G^1)} (\cong \cH(\C))$ with continuous inclusion.

\item[$(ii)$] Let $M$ be log-convex and normalized. Assume that
\[
\mathcal{E}_{\{M\}}(\RR)\subseteq \mathcal{E}_{(G^1)}(\RR) (\cong \cH(\C))
\]
holds (as sets), then $\lim_{p\rightarrow+\infty}(m_p)^{1/p}=0$ follows.
In particular, this implication holds for any $M\in\hyperlink{LCset}{\mathcal{LC}}$.
\end{itemize}
\end{lemma}

Moreover, in the situation of Lemma \ref{smallclasseslemma} the inclusion always has to be strict.
Thus spaces $\cE_{[M]}$ for sequences with $\lim_{p\rightarrow+\infty}m_p^{1/p}=0$ form classes of entire functions. Subsequently, we show that those spaces are weighted classes of entire functions and the weight is given by the \emph{associated weight function} of the \emph{conjugate weight sequence}. We thoroughly define and investigate those terms in the following sections. We remark that the definition of the conjugate sequence has been inspired by the Gevrey case treated by M. Markin, see Example \ref{Gevreyexample} and Section \ref{sec:Markinascor}.

\subsection{Conjugate weight sequence}\label{conjugatesect}
Let $M\in\RR_{>0}^{\NN}$, then we define the {\itshape conjugate sequence} $M^{*}=(M^{*}_p)_{p\in\NN}$ by
\begin{equation}\label{Mstardef}
M^{*}_p:=\frac{p!}{M_p}=\frac{1}{m_p},\;\;\;p\in\NN,
\end{equation}
i.e. $M^{*}:=m^{-1}$ for short. Hence, for $p\ge 1$ the quotients $\mu^{*}=(\mu^{*}_p)_p$ are given by
\begin{equation}\label{Mstardefquot}
\mu^{*}_p:=\frac{M^{*}_p}{M^{*}_{p-1}}=\frac{m_{p-1}}{m_p}=\frac{p!M_{p-1}}{(p-1)!M_p}=\frac{p}{\mu_p},
\end{equation}
and we set $\mu^{*}_0:=1$. By these formulas it is immediate that there is a one-to-one correspondence between $M$ and $M^{*}$.

\subsection{Properties of conjugate weight sequences}
\label{sec:propertiesconjugate}
We summarize some immediate consequences for $M^{*}$. Let $M,N\in\RR_{>0}^{\NN}$ be given.

\begin{itemize}
\item[$(i)$] First, we immediately have
$$\forall\;p\in\NN:\;\;\;M^{**}_p=M_p,\hspace{15pt}M^{*}_p\cdot M_p=p!,$$
i.e.
$$M^{**}\equiv M,\hspace{30pt}M^{*}\cdot M\equiv G^1.$$
Moreover, (see also the subsequent Lemma \ref{Mstarfixpoint})
$$M^{*}\hyperlink{mpreceq}{\preceq}M\Longleftrightarrow G^{1/2}\hyperlink{mpreceq}{\preceq} M,\hspace{15pt}M\hyperlink{mpreceq}{\preceq}M^{*}\Longleftrightarrow M\hyperlink{mpreceq}{\preceq}G^{1/2},$$
and alternatively the relation \hyperlink{mpreceq}{$\preceq$} can be replaced by $\le$. We also get $M^{*}_0=M_0^{-1}$, i.e. $M^{*}$ is normalized if and only if $1=M_0\ge M_1$.

\item[$(ii)$] $M\hyperlink{mpreceq}{\preceq}N$ holds if and only if $N^{*}\hyperlink{mpreceq}{\preceq}M^{*}$ and so $M\hyperlink{approx}{\approx}N$ if and only if $M^{*}\hyperlink{approx}{\approx}N^{*}$.

\item[$(iii)$] We get the following:

\begin{itemize}
\item[$(*)$] $\lim_{p\rightarrow+\infty}(M^{*}_p)^{1/p}=+\infty$ holds if and only if $\lim_{p\rightarrow+\infty}(m_p)^{1/p}=0$ and this implies $\mathcal{E}_{\{M\}}\subseteq\mathcal{E}_{(G^1)}$ (with strict inclusion). If in addition $M$ is log-convex (and normalized) then all three assertions are equivalent, see Lemma \ref{smallclasseslemma}.

\item[$(*)$] If $\lim_{p\rightarrow+\infty}(M_p)^{1/p}=+\infty$, then by $\mu^{*}_p/p=\frac{1}{\mu_p}$, \eqref{mucompare} and Stirling's formula we get both $\lim_{p\rightarrow+\infty}\mu^{*}_p/p=0$ and $\lim_{p\rightarrow+\infty}(m^{*}_p)^{1/p}=0$.

\item[$(*)$] $\lim_{p\rightarrow+\infty}(m^{*}_p)^{1/p}=+\infty$ holds if and only if $\lim_{p\rightarrow+\infty}(M_p)^{1/p}=0$.
\end{itemize}

\item[$(iv)$] $M^{*}$ is log-convex, i.e. $\mu^{*}_{p+1}\ge\mu^{*}_p$ for all $p\in\NN_{>0}$, if and only if $m$ is {\itshape log-concave}, i.e.
    \begin{equation}\label{logconcave}
    \forall\;p\in\NN_{>0}:\;\;\;m_p^2\ge m_{p-1}m_{p+1}\Longleftrightarrow\mu^{*}_{p+1}\ge\mu^{*}_p,
    \end{equation}
   which in turn is equivalent to the map $p\mapsto\frac{\mu_p}{p}$ being non-increasing.

   Analogously as in \cite[Lemma 2.0.4]{diploma} we get: If a sequence $S\in\RR_{>0}^{\NN}$ is log-concave and satisfies $S_0=1$, then the mapping $p\mapsto(S_p)^{1/p}$ is non-increasing.

   Consequently, if $M^{*}$ is log-convex and if $1=M^{*}_0=m_0=M_0$, then $p\mapsto(m_p)^{1/p}$ is non-increasing.

\item[$(v)$] If $M$ is log-convex (and having $M_0=1$), then $M^{*}$ has \hyperlink{mg}{$(\on{mg})$}: In this case by \cite[Lemma 2.0.6]{diploma} for all $p,q\in\NN$ we get $M_pM_q\le M_{p+q}\Leftrightarrow m_pm_q\le\frac{(p+q)!}{p!q!}m_{p+q}$ and so $m_pm_q\le 2^{p+q}m_{p+q}$. Hence $M^{*}_{p+q}\le 2^{p+q}M^{*}_pM^{*}_q$ holds true.

\item[$(vi)$] $M^{*}$ has \hyperlink{dc}{$(\on{dc})$} if and only if $\mu^{*}_p\le A^p\Leftrightarrow\frac{p}{\mu_p}\le A^p$, so if and only if
    \begin{equation*}\label{dualderivation}
    \exists\;A\ge 1\;\forall\;p\in\NN:\;\;\;\mu_p\ge\frac{p}{A^p},
    \end{equation*}
    which can be considered as ``dual derivation closedness''. Note that this property is preserved under the mapping $(M_p)_p\mapsto(M_pp!^s)_p$, $s\in\RR$ arbitrary, and it is mild: $\liminf_{p\rightarrow+\infty}\mu_p/p>0$ is sufficient to conclude.

\item[$(vii)$] $M^{*}$ has \hyperlink{beta1}{$(\beta_1)$}, i.e. $\liminf_{p\rightarrow+\infty}\frac{\mu^{*}_{Qp}}{\mu^{*}_p}>Q$ for some $Q\in\NN_{\ge 2}$, if and only if $\liminf_{p\rightarrow+\infty}\frac{\mu_{p}}{\mu_{Qp}}>1$; similarly $M^{*}$ has \hyperlink{beta3}{$(\beta_3)$} if and only if $\liminf_{p\rightarrow+\infty}\frac{\mu_p}{\mu_{Qp}}>\frac{1}{Q}$.
\end{itemize}

Using those insights, we may conclude the following.

\begin{lemma}\label{MstarLC}
Let $M\in\RR_{>0}^{\NN}$ be given with $1=M_0\ge M_1$ and let $M^{*}$ be the conjugate sequence defined via \eqref{Mstardef}. Then:
\begin{itemize}
\item[$(a)$] $M^{*}\in\hyperlink{LCset}{\mathcal{LC}}$ if and only if $m$ is log-concave and $\lim_{p\rightarrow+\infty}(m_p)^{1/p}=0$.

\item[$(b)$] $M^{*}\in\hyperlink{LCset}{\mathcal{LC}}$ implies $\mathcal{E}_{\{M\}}\subseteq\mathcal{E}_{(G^1)}$ with strict inclusion.

\item[$(c)$] If in addition $M$ is log-convex with $1=M_0=M_1$, then the inclusion $\mathcal{E}_{\{M\}}(\R)\subseteq\mathcal{E}_{(G^1)}(\R)$ gives $\lim_{p\rightarrow+\infty}(M^{*}_p)^{1/p}=+\infty$. Moreover, $M^{*}$ has moderate growth.
\end{itemize}
\end{lemma}

\begin{remark}\label{almoststuff}
Let $M\in\RR_{>0}^{\NN}$ be given and we comment on the log-concavity and related conditions (for the sequence $m$):

\begin{itemize}
\item[$(a)$] If $m$ is not log-concave but satisfies
$$\exists\;H\ge 1\;\forall\;1\le p\le q:\;\;\;\frac{\mu_q}{q}\le H\frac{\mu_p}{p},$$
i.e. the sequence $(\mu_p/p)_{p\in\NN_{>0}}$ is {\itshape almost decreasing,} then the sequence $L$ defined in terms of the corresponding quotient sequence $\lambda=(\lambda_p)_{p\in\NN}$ given by
\begin{equation}\label{almoststuffequ}
\lambda_p:=H^{-1}p\sup_{q\ge p}\frac{\mu_q}{q},\;\;\;p\ge 1,\hspace{20pt}\lambda_0:=1,
\end{equation}
satisfies
\begin{equation}\label{almoststuffequ1}
\forall\;p\ge 1:\;\;\;H^{-1}\frac{\mu_p}{p}\le\frac{\lambda_p}{p}\le\frac{\mu_p}{p}.
\end{equation}
Then we get:
\begin{itemize}
\item[$(i)$] $L$ and $M$ are equivalent and so $L^{*}$ is equivalent to $M^{*}$, too.

\item[$(ii)$] $p\mapsto\frac{\lambda_p}{p}$ is non-increasing, i.e. $l$ is log-concave, and so $L^{*}$ is log-convex.

\item[$(iii)$] If $1=M_0\ge M_1$, i.e. if $\mu_1\le 1$, then $1=L_0\ge L_1$ is valid since $L_1=\lambda_1\le\mu_1\le 1$ holds true. Thus $L^{*}$ is normalized.

\item[$(iv)$] $\lim_{p\rightarrow+\infty}(m_p)^{1/p}=0$ if and only if $\lim_{p\rightarrow+\infty}(l_p)^{1/p}=0$ (with $l_p:=L_p/p!$).

\item[$(v)$] Finally, if $M$ is log-convex, then $L$ shares this property: We have $\lambda_p\le\lambda_{p+1}$ if and only if $p\sup_{q\ge p}\frac{\mu_q}{q}\le(p+1)\sup_{q\ge p+1}\frac{\mu_q}{q}$ for all $p\ge 1$. When $p\ge 1$ is fixed, then clearly $p\frac{\mu_q}{q}\le(p+1)\frac{\mu_q}{q}$ for all $q\ge p+1$. If $q=p$, then
$$p\frac{\mu_q}{q}=\mu_p\le\mu_{p+1}=(p+1)\frac{\mu_{p+1}}{p+1}\le(p+1)\sup_{q\ge p+1}\frac{\mu_q}{q},$$
and so the desired inequality is verified.
\end{itemize}
Summarizing, if $M\in\RR_{>0}^{\NN}$ satisfies $1=M_0\ge M_1$ and $\lim_{p\rightarrow+\infty}(m_p)^{1/p}=0$, then $L^{*}\in\hyperlink{LCset}{\mathcal{LC}}$, see $(a)$ in Lemma \ref{MstarLC}. If $M$ is in addition log-convex, then $L$ has this property too.

The definition \eqref{almoststuffequ} is motivated by \cite[Lemma 8]{whitneyextensionmixedweightfunctionII} and \cite[Prop. 4.15]{JimenezGarridoSanz}.

\item[$(b)$] If $m$ is log-concave, then for any $s\ge 0$ also the sequence $(m_p/p!^s)_{p\in\NN}$ is log-concave because the mapping $p\mapsto\frac{\mu_p}{p^s}$ is still non-increasing (see \eqref{logconcave}). However, for the sequence ($p!^sm_p)_{p\in\NN}$ this is not clear in general.
\end{itemize}
\end{remark}

\begin{example}\label{Gevreyexample}
Let $M\equiv G^s$ for some $0\le s<1$, see \cite{MarkinIII}. (In fact in \cite{MarkinIII} instead of $G^s$ the sequence $(p^{ps})_{p\in\NN}$ is treated but which is equivalent to $G^s$ by Stirling's formula.) Then $m\equiv G^{s-1}$ with $-1\le s-1<0$ and so $m$ corresponds to a Gevrey-sequence with negative index. We get $\lim_{p\rightarrow+\infty}(m_p)^{1/p}=0$ and $m$ is log-concave. Moreover $M^{*}\equiv G^{1-s}$ and so clearly $M^{*}\in\hyperlink{LCset}{\mathcal{LC}}$.
\end{example}

In particular, if $s=\frac{1}{2}$ then $(G^{\frac{1}{2}})^{*}=G^{\frac{1}{2}}$ and we prove the following statement which underlines the importance of $G^{\frac{1}{2}}$ (up to equivalence of sequences) w.r.t. the action $M\mapsto M^{*}$.

\begin{lemma}\label{Mstarfixpoint}
Let $M\in\RR_{>0}^{\NN}$ be given. Then the following are equivalent:
\begin{itemize}
\item[$(i)$] We have $M\hyperlink{mpreceq}{\preceq}M^{*}$.

\item[$(ii)$] We have
$$\exists\;C,h\ge 1\;\forall\;p\in\NN:\;\;\;M_p^2\le Ch^pp!,$$
i.e. $M\hyperlink{mpreceq}{\preceq}G^{1/2}$.

\item[$(iii)$] We have $G^{1/2}\hyperlink{mpreceq}{\preceq}M^{*}$.
\end{itemize}
The analogous equivalences are valid if $M^{*}\hyperlink{mpreceq}{\preceq}M$ resp. if relation \hyperlink{mpreceq}{$\preceq$} is replaced by $\le$. Thus $M\hyperlink{approx}{\approx}M^{*}$ if and only if $M\hyperlink{approx}{\approx}G^{1/2}$ and $M=M^{*}$ if and only if $M=G^{1/2}=M^{*}$.

In particular, $G^{1/2}=(G^{1/2})^{*}$ holds true.
\end{lemma}

\begin{proof}
The equivalences follow immediately from the definition of $M^{*}$ in \eqref{Mstardef}.
\end{proof}

\subsection{Associated weight function}\label{assoweightfctset}
Let $M\in\RR_{>0}^{\NN}$ (with $M_0=1$), then the {\itshape associated function} $\omega_M: \RR_{\ge 0}\rightarrow\RR\cup\{+\infty\}$ is defined by
\begin{equation*}\label{assofunc}
\omega_M(t):=\sup_{p\in\NN}\log\left(\frac{t^p}{M_p}\right)\;\;\;\text{for}\;t>0,\hspace{30pt}\omega_M(0):=0.
\end{equation*}
For an abstract introduction of the associated function we refer to \cite[Chapitre I]{mandelbrojtbook}, see also \cite[Definition 3.1]{Komatsu73}. If $\liminf_{p\rightarrow+\infty}(M_p)^{1/p}>0$, then $\omega_M(t)=0$ for sufficiently small $t$, since $\log\left(\frac{t^p}{M_p}\right)<0\Leftrightarrow t<(M_p)^{1/p}$ holds for all $p\in\NN_{>0}$. Moreover under this assumption $t\mapsto\omega_M(t)$ is a continuous non-decreasing function, which is convex in the variable $\log(t)$ and tends faster to infinity than any $\log(t^p)$, $p\ge 1$, as $t\rightarrow+\infty$. $\lim_{p\rightarrow+\infty}(M_p)^{1/p}=+\infty$ implies that $\omega_M(t)<+\infty$ for each $t>0$ and which shall be considered as a basic assumption for defining $\omega_M$.\vspace{6pt}

Given $M\in\hyperlink{LCset}{\mathcal{LC}}$, then by \cite[1.8 III]{mandelbrojtbook} we get that $\omega_M(t)=0$ on $[0,\mu_1]$.\vspace{6pt}

Finally note that for $M\in\hyperlink{LCset}{\mathcal{LC}}$ we have $\lim_{p\rightarrow+\infty}\mu_p=+\infty$, see e.g. \cite[p. 104]{compositionpaper}.



\section{Ultradifferentiable classes as weighted spaces of entire functions}\label{ultradiffvsweightedentiresect}

In Section \ref{sec:ultraasentire}, we saw that ultradifferentiable classes $\cE_{[M]}$ with $\lim_{p\rightarrow+\infty}m_p^{1/p}=0$ are classes of entire functions. Now we go further and identify those classes with weighted spaces of entire functions, where the weight is given by the associated weight function of the conjugate weight sequence $M^{*}$. To this end, let us first recall some notation already introduced in \cite{inclusion} (to be precise, in \cite{inclusion} the  weighted spaces of entire functions have only been defined for the codomain $\C$, but everything can be done completely analogously for $H$ instead of $\C$): Let $H$ be a Hilbert space and let $v:[0,+\oo) \rightarrow (0,+\oo)$ be a weight function, i.e. $v$ is
\begin{itemize}
\item[$(*)$] continuous,

\item[$(*)$] non-increasing and

\item[$(*)$] rapidly decreasing, i.e. $\lim_{t\rightarrow+\infty}t^kv(t)=0$ for all $k\ge 0$.
\end{itemize}
Then introduce the space
\[
\cH^\infty_v(\C,H):=\{ f \in \cH(\C,H): \|f\|_v:= \sup_{z \in \C} \|f(z)\|v(|z|) < +\infty\}.
\]
We shall assume w.l.o.g. that $v$ is \emph{normalized}, i.e. $v(t)=1$ for $t \in [0,1]$ (if this is not the case one can always switch to another normalized weight $w$ with $\cH^\infty_v(\C,H) = \cH^\infty_w(\C,H)$).\par
In the next step we consider {\itshape weight systems,} see \cite[Sect. 2.2]{inclusion} for more details. For a non-increasing sequence of weights $\underline{\cV}=(v_n)_{n\in\NN_{>0}}$, i.e. $v_n\ge v_{n+1}$ for all $n$, we define the (LB)-space
\[
\cH_{\underline{\cV}}^\oo(\C,H):=\varinjlim_{n\in \N_{>0}} \cH^\infty_{v_n}(\C,H),
\]
and for a non-decreasing sequence of weights $\overline{\cV}=(v_n)_{n\in\NN_{>0}}$, i.e. $v_n\le v_{n+1}$ for all $n$, we define the Fr\'echet space
\[
\cH_{\overline{\cV}}^\oo(\C,H):=\varprojlim_{n\in \N_{>0}} \cH^\infty_{v_n}(\C,H).
\]
\begin{remark}
    In \cite{inclusion}, the spaces are denoted by $H^\infty_v(\C)$ instead of $\cH^\infty_v(\C,\C)$. We use $\cH$ in order to avoid any confusion with the Hilbert space $H$. In addition,  $\cH^\oo_v(\C)$ shall denote $\cH^\oo_v(\C,\C)$.
\end{remark}

The following Lemma can be used to infer statements for $\cH^\oo_v(\C,H)$ from the respective statements for $\cH^\oo_v(\C)$.

\begin{lemma}
    \label{lem:liftlemma}
    Let $H$ be a (complex) Hilbert space and $v$ be a weight. Then
    \[
    f \in \cH^\oo_v(\C,H) ~\Leftrightarrow z \mapsto \langle f(z), y \rangle \in \cH^\oo_v(\C) \text{ for all } y \in H.
    \]
\end{lemma}

\begin{proof}
   For the non-trivial part, take some $f \in \cH(\C,H)$ such that $|\langle f(z), y \rangle| \le C_y v(|z|)$ for every $y \in H$. Then this just means that $\{\frac{f(z)}{v(|z|)}:~z \in \C\}$ is weakly bounded (in $H$) which implies boundedness and this just means that $f \in \cH_v(\C,H)$.
\end{proof}

\begin{remark}
    Of course the same argument holds for a family of weights $\overline{\cV}$ or $\underline{\cV}$.
\end{remark}

For a given weight $v$ and $c>0$, we shall write $v_c(t):= v(ct)$ and $v^c(t):=v(t)^c$, and set
\[
\underline{\cV}_\fc = (v_c)_{c \in \N_{>0}}, \text{ and } \overline{\cV}_\fc = (v_{1/c})_{c \in \N_{>0}},
\]
and
\[
\underline{\cV}^\fc = (v^c)_{c \in \N_{>0}}, \text{ and } \overline{\cV}^\fc = (v^{1/c})_{c \in \N_{>0}},
\]
in particular $\underline{\cV}_\fc$ and $\underline{\cV}^\fc$ are non-increasing, and $\overline{\cV}_\fc$ and $\overline{\cV}^\fc$ are non-decreasing sequences of weights, see again \cite[Sect. 2.2]{inclusion}.\par

Let $M\in\RR_{>0}^{\NN}$ be given with $M_0=1$, such that $M$ is \hyperlink{lc}{$(\on{lc})$} and satisfies $\lim_{p\rightarrow+\infty}(M_p)^{1/p}=+\infty$ (see \cite[Def. 2.4, Rem. 2.6]{inclusion}). Then we denote by $\underline{\cM}_\fc, \underline{\cM}^\fc, \overline{\cM}_\fc,$ and $\overline{\cM}^\fc$ the respective sequences of weights defined by choosing $v(t):= v_M(t):= e^{-\om_M(t)}$ (see \cite[Rem. 2.7]{inclusion}). If we write $\underline{\cN}_\fc, \underline{\cN}^\fc, \overline{\cN}_\fc,$ and $\overline{\cN}^\fc$ we mean the respective definition for another weight sequence $N$. Finally, we write (of course) $\underline{\cM^*}_\fc,\dots$ for the systems corresponding to the conjugate sequence $M^*$.

\begin{theorem}
\label{thm:main}
Let $M \in \R^\N_{>0}$ with $M_0 = 1 \ge M_1$  be given such that $\lim_{p\rightarrow+\infty}(m_p)^{1/p}=0$ and $m$ is log-concave. Let $I\subseteq\RR$ be an interval, then
\[
E:\cE_{\{M\}}(I,H) \rightarrow \cH^\oo_{\underline{\cM^*}_\fc}(\C,H), \quad f \mapsto E(f):=\sum_{k = 0}^{+\oo}\frac{f^{(k)}(x_0)}{k!}(z-x_0)^k
\]
is an isomorphism (of locally convex spaces) for any fixed $x_0\in I$. Moreover, with the same definition for $E$, also
\[
E:\cE_{(M)}(I,H) \rightarrow \cH^\oo_{\overline{\cM^*}_\fc}(\C,H)
\]
is an isomorphism.
\end{theorem}

\begin{remark}\label{mainthmremark}
Before proving this main statement we give the following observations:

\begin{itemize}
    \item[$(i)$] By Lemma \ref{MstarLC}, the assumptions on $M$ imply $M^{*}\in\hyperlink{LCset}{\mathcal{LC}}$. It is easy to check that any \emph{small Gevrey} class, i.e. choosing $M_j= j!^\al$ for some $\al \in[0,1)$, satisfies the assumptions of Theorem \ref{thm:main}.

 \item[$(ii)$] Moreover, we comment in detail on the basic requirements for the sequence $M$ in Theorem \ref{thm:main}:

 \begin{itemize}
 \item[$(*)$] Note that both assumptions $M_0=1\ge M_1$ and log-concavity of $m$ are {\itshape not preserved} under equivalence of weight sequences.

 On the other hand, both isomorphisms in Theorem \ref{thm:main} are clearly preserved under equivalence: Equivalent sequences yield the same ultradifferentiable function classes, equivalent conjugate sequences (recall $(ii)$ in Section \ref{sec:propertiesconjugate}) and finally (by definition) also the same weighted entire function classes, see \cite[Prop. 3.8]{inclusion}.

\item[$(*)$] Thus we can assume more generally that $M$ is equivalent to $L\in\RR_{>0}^{\NN}$ such that $L_0 = 1 \ge L_1$, $\lim_{p\rightarrow+\infty}(l_p)^{1/p}=0$, and $l$ is log-concave. In this situation we replace in the proof below $M$ by $L$, $m$ by $l$ and $M^{*}$ by $L^{*}$. Recall that the log-concavity for $l$ can be ensured, e.g., if $(\mu_p/p)_{p\in\NN_{>0}}$ is almost decreasing, see Remark \ref{almoststuff}.

\item[$(*)$] Finally, note the following: Assume that $M$ is equivalent to $L\in\RR_{>0}^{\NN}$ such that $L_0=1\ge L_1$ and $\lim_{p\rightarrow+\infty}(l_p)^{1/p}=0$ but none of the sequences $L$ being equivalent to $M$ has the property that $l$ is log-concave. Thus log-convexity for $L^{*}$ fails for any $L$ being equivalent to $M$. Then both $\cH^\oo_{\underline{\cL^*}_\fc}(\C,H)$ and $\cH^\oo_{\overline{\cL^*}_\fc}(\C,H)$ coincide with the respective classes when $L^{*}$ is replaced by its log-convex minorant $(L^{*})^{\on{lc}}$, see \cite[Rem. 2.6]{inclusion}. In this situation the first part of the proof stays valid; i.e. the operator $E$ is still continuous. But the second part fails in general; more precisely for the equality just below \eqref{eq:cauchy} in the subsequent proof, the log-convexity of the appearing conjugate sequence is indispensable and without this property we can only bound $F^{(n)}$ in terms of $\frac{n!}{(L_n^{*})^{\on{lc}}}=:\overline{L}_n\ge L_n$.
\end{itemize}
\end{itemize}
\end{remark}

\begin{proof}[Proof of Theorem \ref{thm:main}]
We start with the Roumieu case and assume w.l.o.g. that $x_0=0$. Let us take $f \in \cE_{M,h}(K,H)$ for some compact set $K \subset\subset I$ with $0\in K$ and some $h>0$, i.e. there is $A(=\|f\|_{M,K,h})$ such that for all $x \in K$ and all $k \in \N$ we have
\[
\|f^{(k)}(x)\| \le A h^k M_k.
\]
Then we infer immediately that
\[
\|E(f)(z)\| \le A \sum_{k = 0}^{+\oo}\frac{h^kM_k}{k!}|z|^k=A\sum_{k = 0}^{+\oo}\frac{h^k}{M^{*}_k}|z|^k \le 2A \exp(\om_{M^*}(2h|z|)).
\]

Therefore $E$ maps $\cE_{M,h}(K,H)$ continuously into $\cH^\oo_{v_{M^*, 2h}}(\C,H)$ and this immediately implies continuity of $E$ as a mapping defined on the inductive limit with respect to $h$. \par
In the Beurling case a function $f \in \cE_{(M)}(I, H)$ lies in $\cE_{M,h}(K,H)$ for any $h>0$, and thus the above reasoning immediately gives that $E$ is continuous as a mapping into $\mathcal{H}^\oo_{\overline{\cM^*}_\fc}(\C,H)$.\par

Let us now show continuity of the inverse mapping, which is clearly given by restricting an entire function to the interval $I$. Take some $F \in \cH^\oo_{v_{M^*, k}}(\C,H)$, then
\[
\|F(z)\|\le A e^{\om_{M^*}(k|z|)}
\]
for $A=\|F\|_{v_{M^*,k}}>0$. Consider an arbitrary $K \subset\subset I$ and let $R\ge 1$ be such that $K \subset [-R,R]$. Then take $r\ge 2R$, which ensures that $K + B(0,r) \subset B(0,2r)$ and where $B(0,r)$ denotes the ball around $0$ of radius $r$. Then by the Cauchy estimates we infer for such $r$ and all $x \in K$ and $n\in\NN$
\begin{equation}
\label{eq:cauchy}
\|F^{(n)}(x)\|\le An!\frac{e^{\om_{M^*}(2kr)}}{r^n}.
\end{equation}
Since $e^{\om_{M^*}(r)}=\frac{r^n}{M^*_n}$ for $r \in [\mu^*_n, \mu^*_{n+1})$ (see e.g. \cite[1.8 III]{mandelbrojtbook}), we may plug in some $r \in [\mu^*_n/(2k), \mu^*_{n+1}/(2k))$ in \eqref{eq:cauchy}; for all $n$ large enough such that $\mu^*_n/(2k) \ge 2R$ (thus depending on chosen compact $K$) and which is possible since $M^* \in \cL\cC$ and so $\lim_{n\rightarrow+\infty}\mu^*_n=+\infty$. Hence we get
\[
\|F^{(n)}(x)\|\le An!\frac{(2kr)^n}{r^n M_n^*}=A (2k)^nM_n.
\]
For the remaining (finitely, say $n_0$) many integers $n$ with $\mu_n^*/(2k) < 2R$, we can estimate
\[
\|F^{(n)}(x)\|\le CA (2k)^n M_n
\]
where, e.g.  $C=n_0!e^{\om_{M^*}(2kR)}$. Altogether we have shown
\[
\|F|_{I}\|_{M,K,2k} \le C \|F\|_{v_{M^*,k}}
\]
which proves continuity of the inverse mapping in both the Roumieu and the Beurling case.
\end{proof}

\subsection{Comparison of $\cH^\oo_{\underline{\cM^*}_\fc}$ and $\cH^\oo_{\underline{\cM^*}^\fc}$ (resp. $\cH^\oo_{\overline{\cM^*}_\fc}$ and $\cH^\oo_{\overline{\cM^*}^\fc}$)}

Let us quickly recall a recent result characterizing the equality of the two different types of weighted spaces of entire functions, see \cite[Thm. 5.4]{inclusion}.
To this end we need one more condition for $M$:
\begin{equation}
    \label{eq:Momega1}
    \E L\in \N_{>0}:\quad \liminf_{j \rightarrow+\oo} \frac{(M_{Lj})^{1/(Lj)}}{(M_j)^{1/j}}>1.
\end{equation}
In \cite[Thm. 3.1]{subaddlike} it has been shown that $M\in\hyperlink{LCset}{\mathcal{LC}}$ has \eqref{eq:Momega1} if and only if
    \begin{equation}\label{om1}
    \om_M(2t)=O(\om_M(t)) \text{ as } t \rightarrow+\oo.
    \end{equation}

\begin{lemma}
\label{lem:weightedentireiso}
Let $M \in\hyperlink{LCset}{\mathcal{LC}}$. Then the following statements are equivalent:
\begin{itemize}
    \item[$(i)$] $M$ has \hyperlink{mg}{$(\on{mg})$} and satisfies \eqref{eq:Momega1},
    \item[$(ii)$] $\cH^\oo_{\underline{\cM}_\fc}(\C,H) \cong \cH^\oo_{\underline{\cM}^\fc}(\C,H)$,
    \item[$(iii)$] $\cH^\oo_{\overline{\cM}_\fc}(\C,H) \cong \cH^\oo_{\overline{\cM}^\fc}(\C,H)$.
\end{itemize}
\end{lemma}

\begin{proof}
    In \cite[Thm. 5.4]{inclusion}, the result is shown for $H = \C$. In order to get that $(i)$ implies $(ii)$ and $(iii)$ the proof of \cite[Thm. 5.4]{inclusion} can be repeated and only the appearances of $|\cdot|$ (the absolute value in $\C$) have to be substituted by $\|\cdot\|$ (the norm in the Hilbert space $H$).\par
    In order to get the other implications, i.e. that $(ii)$ resp. $(iii)$ implies $(i)$, note that the respective equality in the Hilbert space-valued case implies the equality for the $\C$-valued case by observing that $f \in \cH^\oo_v(\C) (=\cH^\oo_v(\C,\C))$ if and only if for any $0 \neq x \in H$ we have $z \mapsto f(z)x \in \cH^\oo_v(\C,H)$. Therefore we may apply \cite[Thm. 5.4]{inclusion} and infer $(i)$.
\end{proof}

Together with results from Section \ref{sec:propertiesconjugate}, we derive the following.

\begin{corollary}
\label{cor:weightedentireiso}
    Let $M \in \R^\N_{>0}$ be given and assume the following:

\begin{itemize}
    \item[$(*)$] $M$ is log-convex with $1=M_0=M_1$ (i.e. both normalization and $1 = M_0 \ge M_1$),

    \item[$(*)$]  $\lim_{p\rightarrow+\infty}m_p^{1/p}=0$,

\item[$(*)$] $m$ is log-concave, and finally

\item[$(*)$] for some $Q\in\NN_{\ge 2}$ we have $\liminf_{p \rightarrow+\infty} \frac{\mu_p}{\mu_{Qp}} > \frac{1}{Q}$.
\end{itemize}

Then
    \[
    \cH^\oo_{\underline{\cM^*}_\fc}(\C,H) \cong \cH^\oo_{\underline{\cM^*}^\fc}(\C,H), \quad  \cH^\oo_{\overline{\cM^*}_\fc}(\C,H) \cong \cH^\oo_{\overline{\cM^*}^\fc}(\C,H),
    \]
    and $E$ is an isomorphism between $\cE_{\{M\}}(I,H)$ and $\cH^\oo_{\underline{\cM^*}^\fc}(\C,H)$ resp. between $\cE_{(M)}(I,H)$ and  $\cH^\oo_{\overline{\cM^*}^\fc}(\C,H)$.
\end{corollary}

\begin{proof}
    By $(v)$ in Section \ref{sec:propertiesconjugate} it follows that $M^*$ has \hyperlink{mg}{$(\on{mg})$}. By $(vii)$ from Section \ref{sec:propertiesconjugate} we infer that $M^*$ has $(\be_3)$ and thus \cite[Prop. 3.4]{subaddlike} gives that $M^*$ has \eqref{eq:Momega1}. Finally observe that $M^*\in\hyperlink{LCset}{\mathcal{LC}}$ holds true: $\lim_{p\rightarrow+\infty}m_p^{1/p}=0$ implies $\lim_{p\rightarrow+\infty}(M^*_p)^{1/p}=+\infty$ (see $(iii)$ in Section \ref{sec:propertiesconjugate}), log-convexity of $M^*$ follows from log-concavity of $m$ (see $(iv)$ in Section \ref{sec:propertiesconjugate}) and normalization of $M^*$ is immediate. Thus we may apply Lemma \ref{lem:weightedentireiso} to $M^*$. The rest follows from Theorem \ref{thm:main}.
\end{proof}

\begin{remark}
 Observe that the conditions of Lemma \ref{lem:weightedentireiso} hold if and only if $\cE_{[M^*]} \cong \cE_{[\om_{M^*}]}$, cf.
 \cite[Thm. 14]{BonetMeiseMelikhov07}, \cite[Sect. 5]{compositionpaper} and \cite[Prop. 3.4]{subaddlike}.

 Note also that Corollary \ref{cor:weightedentireiso} applies, in particular, to all small Gevrey sequences $G^{\al}$, $0\le\alpha<1$, see the next Section for its importance.
\end{remark}

\subsection{A result by Markin as a Corollary of Theorem \ref{thm:main}}
\label{sec:Markinascor}

One of Markin's core results in \cite{MarkinIII}, Lemma 3.1, shows, in our setting the following: For any $\al \in [0,1)$ and $M^\al_j:=j^{j\al}$, which is equivalent to $G^\al_j= j!^\al$ (i.e. the small Gevrey sequence of order $\al$) and with $v(t):= e^{-t^{1/(1-\al)}}$ we obtain that
\[
E: \cE_{\{G^\al\}}(I,H)\rightarrow \cH^\oo_{\underline{\cV}^\fc}(\C,H)
\]
is an isomorphism of locally convex vector spaces; and mutatis mutandis the same holds in the respective Beurling case. With our preparation, this now is a corollary of Theorem \ref{thm:main} together with the following observations:
\begin{itemize}
    \item Corollary \ref{cor:weightedentireiso} applies to $M=G^\al$,
    \item $(G^\al)^* = G^{1-\al}$,
    \item $\om_{G^{1-\al}} \cong t^{\frac{1}{1-\al}}$, i.e. $\om_{G^{1-\al}}(t) = O( t^{\frac{1}{1-\al}}), ~ t^{\frac{1}{1-\al}} = O(\om_{G^{1-\al}}(t))$ as $t \rightarrow+\infty$.
\end{itemize}

\subsection{Characterization of inclusion relations for small weight sequences}\label{charactsmallsequsection}
In the theory of ultradifferentiable functions the characterization of the inclusion $\mathcal{E}_{[M]}\subseteq\mathcal{E}_{[N]}$ in terms of a growth property expressed in terms of $M$ and $N$ is studied. Summarizing we get the following, e.g. see \cite[Prop. 2.12]{compositionpaper} and the literature citations there; similar/analogous techniques have also been applied to the more general and recent approaches in \cite[Prop. 4.6]{compositionpaper} and \cite[Sect. 4]{PTTvsmatrix}:

\begin{itemize}
\item[$(*)$] If $M,N\in\RR_{>0}^{\NN}$ with $M\hyperlink{mpreceq}{\preceq}N$, then $\mathcal{E}_{\{M\}}\subseteq\mathcal{E}_{\{N\}}$ and $\mathcal{E}_{(M)}\subseteq\mathcal{E}_{(N)}$ with continuous inclusion.

\item[$(*)$] If in addition $M$ is normalized and log-convex, then $\mathcal{E}_{\{M\}}(\RR)\subseteq\mathcal{E}_{\{N\}}(\RR)$ (as sets) yields $M\hyperlink{mpreceq}{\preceq}N$.

    If $M,N\in\hyperlink{LCset}{\mathcal{LC}}$, then $\mathcal{E}_{(M)}(\RR)\subseteq\mathcal{E}_{(N)}(\RR)$ (as sets and/or with continuous inclusion, see the proof of \cite[Prop. 4.6]{compositionpaper} and \cite[Prop. 4.5, Rem. 4.6]{PTTvsmatrix}) yields $M\hyperlink{mpreceq}{\preceq}N$.
\end{itemize}

Thus for the necessity of $M\hyperlink{mpreceq}{\preceq}N$ standard regularity and growth assumptions for $M$ are required and so far it is not known what can be said for (small) sequences $M$ ``beyond'' this setting. Via an application of Theorem \ref{thm:main} and main results from \cite{inclusion}, we now may actually prove as a corollary an analogous statement.\par

First let us recall \cite[Thm. 3.14]{inclusion}, where the following characterization is shown (even under formally slightly more general assumptions on the weight $N$, see also \cite[Rem. 2.6]{inclusion}).

\begin{theorem}\label{weightholombysequcharact}
Let $N\in\hyperlink{LCset}{\mathcal{LC}}$ and $M\in\RR_{>0}^{\NN}$ such that $M$ satisfies $M_0=1$ and $\lim_{p\rightarrow+\infty}(M_p)^{1/p}=+\infty$. Then the following are equivalent:
\begin{itemize}
\item[$(a)$] We have $N\hyperlink{preceq}{\preceq}M$.

\item[$(b)$] We have
$$\cH^{\infty}_{\underline{\mathcal{M}}_{\mathfrak{c}}}(\CC)\subseteq \cH^{\infty}_{\underline{\mathcal{N}}_{\mathfrak{c}}}(\CC).$$

\item[$(c)$] We have
$$\cH^{\infty}_{\overline{\mathcal{M}}_{\mathfrak{c}}}(\CC)\subseteq \cH^{\infty}_{\overline{\mathcal{N}}_{\mathfrak{c}}}(\CC).$$
\end{itemize}
\end{theorem}

Thus, by combining Theorem \ref{thm:main} and Theorem \ref{weightholombysequcharact}, which we apply to $N^{*}$ and $M^{*}$, we get the following:

\begin{theorem}\label{equivalencesmallsequthm}
Let $M,N\in\RR_{>0}^{\NN}$ be given and assume that
\begin{itemize}
\item[$(*)$] $1=M_0\ge M_1$ and $1=N_0\ge N_1$,

\item[$(*)$] $\lim_{p\rightarrow+\infty}(m_p)^{1/p}=\lim_{p\rightarrow+\infty}(n_p)^{1/p}=0$,

\item[$(*)$] both $m$ and $n$ are log-concave.
\end{itemize}

Then the following are equivalent:
\begin{itemize}
\item[$(i)$] We have $M\hyperlink{mpreceq}{\preceq}N$.

\item[$(ii)$] We have $\mathcal{E}_{\{M\}}\subseteq\mathcal{E}_{\{N\}}$ with continuous inclusion.

\item[$(iii)$] We have $\mathcal{E}_{(M)}\subseteq\mathcal{E}_{(N)}$ with continuous inclusion.
\end{itemize}
\end{theorem}

\begin{proof}
It remains to prove $(ii),(iii)\Rightarrow(i)$. We use the inclusion in $(ii)$ resp. in $(iii)$ for some compact interval $I$, i.e. $\mathcal{E}_{[M]}(I)\subseteq\mathcal{E}_{[N]}(I)$. Then the characterization shown in Theorem \ref{thm:main} yields $\cH^{\infty}_{\underline{\mathcal{M}^{*}}_{\mathfrak{c}}}(\CC)\subseteq \cH^{\infty}_{\underline{\mathcal{N}^{*}}_{\mathfrak{c}}}(\CC)$ resp. $\cH^{\infty}_{\overline{\mathcal{M}^{*}}_{\mathfrak{c}}}(\CC)\subseteq \cH^{\infty}_{\overline{\mathcal{N}^{*}}_{\mathfrak{c}}}(\CC)$. By the assumptions on $M,N$ we get $M^{*}, N^{*}\in\hyperlink{LCset}{\mathcal{LC}}$ and then Theorem \ref{weightholombysequcharact} gives $N^{*}\hyperlink{preceq}{\preceq}M^{*}$ which is equivalent to $M\hyperlink{preceq}{\preceq}N$ (recall $(ii)$ in Section \ref{sec:propertiesconjugate}) and so $(i)$ is shown.
\end{proof}

\section{A criterion for boundedness of an operator on a Hilbert space}\label{boundednesssection}

The aim of this section is to generalize results by M. Markin from \cite{MarkinI}, \cite{MarkinII}, and \cite{MarkinIII} (obtained within the so-called {\itshape small Gevrey setting}) to a more general weight sequence setting when considering appropriate families of small weight sequences. (In fact, Markin considers instead of $G^{\be}$, $0\le\be<1$, the sequence $M^\be_j:=j^{j\be}$ but which is equivalent to $G^{\be}$ by Stirling's formula. Since equivalence clearly preserves the corresponding function spaces his results immediately transfer to $G^{\be}$ as well.)\vspace{6pt}

Markin studies, for a Hilbert space $H$, and a normal (unbounded) operator $A$ on $H$ the associated evolution equation
\begin{equation}
    \label{eq:evo}
    y'(t) = A y(t)
\end{equation}
and asks whether a priori known smoothness of all solutions of \eqref{eq:evo} yield boundedness of the operator $A$.\par
For a detailed exposition of evolution equations on Hilbert spaces we refer to Chapters 1 (bounded case) and 4 (unbounded case) in \cite{Pazy}.

\subsection{Solutions for bounded operators}

First let us recall quickly the situation for bounded operators $A$. For those, the domain is all of $H$. It is a classical result in this context that every solution $y$ of \eqref{eq:evo} is of the form
\[
y(t)=e^{tA}y_0,
\]
for some $y_0 \in H$, where $e^{tA}:= \sum_{k = 0}^{+\oo} \frac{t^k}{k!}A^k$ and which converges locally uniformly (with respect to $t$) in the norm topology on $B(H)$ (the space of bounded operators on $H$). Moreover, $y$ can be extended to an entire function such that
\[
\|y(z)\|\le Me^{C|z|}
\]
for some constants $M$ and $C$ and all $z\in\CC$.
Thus we may conclude the subsequent statement.
\begin{itemize}
    \item[$(i)$] If $A$ is a bounded operator on $H$, then \emph{each} solution $y$ of \eqref{eq:evo} is an entire function of exponential type.
\end{itemize}
On the other hand we have the following:
\begin{itemize}
    \item [$(ii)$] As outlined by M. Markin in \cite{MarkinI}, \cite{MarkinII} and \cite{MarkinIII}, there exists an unbounded normal operator $A$ (that is actually not bounded on $H$) such that each (weak) solution of \eqref{eq:evo} is an entire function.
\end{itemize}

\subsection{Motivating question}
\label{sec:motivating}
So one may ask whether one can reverse the implication in $(i)$, and if this is possible to what extent one can weaken the assumption of exponential type. From $(ii)$ it is clear that one cannot get completely rid of any additional growth restriction!
\par
Markin does exactly that in \cite{MarkinIII}. Let us first recall his approach and then subsequently considerably extend it.

\subsection{A generalization of Markin's results}

The main result \cite[Thm. 5.1]{MarkinIII} states that if \emph{each} weak solution of \eqref{eq:evo} is in some \emph{small} Gevrey class, i.e. admitting a growth restriction expressed in germs of $G^\al$ with $\alpha<1$, then the operator $A$ is necessarily bounded on $H$.
This is of special interest since, as outlined in Section \ref{sec:Markinascor}, every small Gevrey class can be identified with a weighted class of entire functions. \par
Before we are able to generalize Markin's result we need some definitions:
For a densely defined operator $A$ on $H$, we first set
\[
C^\infty(A):=\bigcap_{n \in \N} D(A^n),
\]
where $D(A^n)$ is the domain of $A^n$, the $n$-fold iteration of $A$. Then put
\[
\cE_{\{M\}}(A):= \{f \in C^\oo(A):~ \E C,h >0~ \A n \in \N ~\|A^nf\|\le C h^n M_n\},
\]
and the corresponding Beurling class is defined by
\[
\cE_{(M)}(A):= \{f \in C^\oo(A):~ \forall\;h >0\;\exists\;C>0~ \A n \in \N ~\|A^nf\|\le C h^n M_n\}.
\]

From \cite[Sect. 1.3]{Gorbachuk} a different description of $\cE_{\{M\}}(A)$ in terms of $E_A$, the spectral measure associated to $A$, can be deduced as follows:
\[
\cE_{\{M\}}(A)=\{ f \in H:~\E t>0~ \int_\C e^{2 \om_M(t|\la|)} \langle dE_A(\la) f,f\rangle <+\infty\},
\]
and
\[
\cE_{(M)}(A)=\{ f \in H:~\A t>0~ \int_\C e^{2\om_M(t|\la|)} \langle dE_A(\la) f,f\rangle <+\infty\}.
\]
Now we have the following result which generalizes \cite[Thm. 3.1]{MarkinII}.

\begin{theorem}
\label{thm:regsolution}
    Let $M\in\RR_{>0}^{\NN}$ be given and $I\subseteq\RR$ a closed interval. Then a solution $y$ of \eqref{eq:evo} belongs to $\cE_{[M]}(I,H)$ if and only if $y(t) \in \cE_{[M]}(A)$ for all $t \in I$. In this case one has $y^{(n)}(t)=A^ny(t)$ for all $t \in I$.
\end{theorem}

\begin{proof}
Let $y$ be a solution of \eqref{eq:evo} such that $y \in \cE_{[M]}(I,H)$. Since $y \in C^\oo(I,H)$, we have by \cite[Prop. 4.1]{MarkinI} that $y^{(n)}(t)=A^ny(t)$ for all $t \in I$ and all $n \in \N$. Therefore
\[
\|A^ny(t)\| = \|y^{(n)}(t)\| \le C h^n M_n,
\]
where $h$ is either in the scope of an existential or universal quantifier depending on the context. This immediately gives that $y(t) \in \cE_{[M]}(A)$ for all $t$.\par
For the converse direction, we argue as in \cite[Proof of Prop. 3.1]{MarkinII} where it is  shown that in this case for any subinterval $[a,b] \subseteq I$
\[
\max_{t \in [a,b]}\|y^{(n)}(t)\| \le \|y^{(n)}(a)\|+\|y^{(n)}(b)\|.
\]
Since again we have $y^{(n)}(t) = A^n y(t)$ this immediately yields $y \in \cE_{[M]}(I,H)$.
\end{proof}

We need one more result generalizing \cite[Lemma 4.1]{MarkinIII} which reads as follows.

\begin{lemma}
\label{lem:Lemma4}
Let $0<\be<+\infty$. If
\[
\bigcup_{0<\be'<\be}\cE_{\{G^{\be'}\}}(A) = \cE_{(G^{\be})}(A),
\]
then the operator $A$ is bounded.
\end{lemma}
Note that in \cite{MarkinIII} the notation $\cE^{[\be]}(A)$ is used instead of $\cE_{[G^\be]}(A)$ (i.e. the respective Gevrey class of order $\be$). Since we have a generalization of \cite[Thm. 5.1]{MarkinIII} as our goal, we only need a generalization of the above Lemma in the case $\be = 1$. So we want to conclude that an operator $A$ on a Hilbert space $H$ is bounded if we can write the \emph{entire} functions corresponding to $A$ (i.e. $\be =1$) as an union of certain smaller Roumieu classes. Note that this statement might seem ``counterintuitive'' when considering the ultradifferentiable classes introduced in Section \ref{ultradiffclasssection}. However, note that the classes in Section \ref{ultradiffclasssection} are defined by using the differential operator which is unbounded.\par

Summarizing, our generalization of Markin's result reads as follows.

\begin{lemma}
    \label{lem:Lemma4General}
    Let $\fF \subseteq \cL\cC$ be a family of sequences such that
     \begin{equation}
     \label{Rmixedom1}
    \forall\;N\in\fF\;\exists\;M\in\fF:\;\;\;\om_M(2t)=O(\om_N(t)) \text{ as } t \rightarrow+\oo,
    \end{equation}
    i.e. a mixed version of \eqref{om1} (of Roumieu-type, see \cite[Sect. 3]{mixedgrowthindex}).

    Suppose there exists $\mathbf{a} = (a_j)\in\RR_{>0}^{\NN}$ with the following properties:
    \begin{itemize}
    \item[$(i)$] we have $\lim_{j \rightarrow+\oo} a_j^{1/j} = 0$,

    \item[$(ii)$] $\mathbf{a}$ is a {\itshape uniform bound} for $\fF$ which means that
    $$\forall\;N \in \fF\;\exists\;C>0\;\forall\;j\in\NN:\;\;\;(N_j/j!=)n_j \le Ca_j.$$
    \end{itemize}
    Then
     \[
    \bigcup_{N \in \fF}\cE_{\{N\}}(A) = \cE_{(G^1)}(A) \text{ as sets }
    \]
    implies that $A$ is bounded.
\end{lemma}

\begin{remark}
\label{rem:Markin}
We gather some comments concerning the previous result:
\begin{itemize}
\item[$(*)$] By choosing $a_j = \frac{1}{\log(j)^j}$, Lemma \ref{lem:Lemma4General} includes Lemma \ref{lem:Lemma4} (with $\be = 1$) as a special case.

\item[$(*)$] Requirements $(i)$ and $(ii)$ in Lemma \ref{lem:Lemma4General} imply that $\lim_{j \rightarrow+\oo} n_j^{1/j} = 0$ for all $N \in \fF$.

\item[$(*)$] If each $N\in\fF$ satisfies \eqref{eq:Momega1}, then \eqref{Rmixedom1} follows with $M=N$.

\item[$(*)$] In \cite[Thm. 3.2]{mixedgrowthindex} condition \eqref{Rmixedom1} has been characterized for one-parameter families (weight matrices, see \cite[Sect. 2.5]{mixedgrowthindex}) in terms of the following requirement:
    $$\exists\;r>1\;\forall\;N\in \fF\;\exists\;M\in \fF\;\exists\;L\in\NN_{>0}:\quad \liminf_{j \rightarrow+\oo} \frac{(M_{Lj})^{1/(Lj)}}{(N_j)^{1/j}}>r,$$
    i.e. a mixed version of \eqref{eq:Momega1}.
\end{itemize}
\end{remark}

Actually we show now that, if $\fF$ consists of a one-parameter family of sequences having some rather mild regularity and growth properties, then it is already possible to find some sequence $\mathbf{a}$ as required in Lemma \ref{lem:Lemma4General}.

\begin{proposition}
\label{prop:uniformbound}
    Let $\fF:=\{N^{(\beta)}\in\RR_{>0}^{\NN}: \beta>0\}$ be a one-parameter family of
sequences $N^{(\beta)}$ which satisfies the following properties:
\begin{itemize}
    \item[$(i)$] $N_0^{(\beta)}=1$ for all $\beta>0$ (normalization),\\
    \item[$(ii)$] $N^{(\beta_1)}\le N^{(\beta_2)}\Leftrightarrow n^{(\beta_1)}\le
    n^{(\beta_2)}$ for all $0<\beta_1\le\beta_2$ (point-wise order),\\
    \item[$(iii)$] $\lim_{j\rightarrow+\infty}(n^{(\beta)}_j)^{1/j}=0$ for each
    $\beta>0$,\\
    \item[$(iv)$] $j\mapsto(n^{(\be)}_j)^{1/j}$ is non-increasing for every $\be >0$,\\
    \item[$(v)$] $\lim_{j\rightarrow+\infty}\left(\frac{N^{(\beta_2)}_j}{N^{(\beta_1)}_j}\right)^{1/j}=\lim_{j\rightarrow+\infty}\left(\frac{n^{(\beta_2)}_j}{n^{(\beta_1)}_j}\right)^{1/j}=+\infty$
for all $0<\beta_1<\beta_2$ (large growth difference between the sequences).
\end{itemize}

Then there exists $\mathbf{a}=(a_j)_j\in\RR_{>0}^{\NN}$ such that
\begin{itemize}
    \item[$(*)$] $j\mapsto(a_j)^{1/j}$ is non-increasing,\\
    \item[$(*)$] $\lim_{j\rightarrow+\infty}(a_j)^{1/j}=0$, and\\
    \item[$(*)$] $\lim_{j\rightarrow+\infty}\left(\frac{a_j}{n_j^{(\beta)}}\right)^{1/j}=+\infty$
for all $\beta>0$.
\end{itemize}
In particular, this implies that there exists a uniform sequence/bound $\mathbf{a}$ for $\fF$ as required in Lemma \ref{lem:Lemma4General}.

In addition, the family $\fF$ satisfies \eqref{Rmixedom1}.
\end{proposition}

{\itshape Note:}
\begin{itemize}
    \item[$(*)$] Requirement $(iv)$ weaker than assuming log-concavity for each $n^{(\be)}$: Together with $(i)$, i.e. $n^{(\be)}_0=1$ (for each $\be$), log-concavity implies $(iv)$; see $(iv)$ in Section \ref{sec:propertiesconjugate}.

\item[$(*)$] Moreover, if $(iv)$ is replaced by assuming that each $n^{(\be)}$ is log-concave and $(i)$ by slightly stronger $n^{(\beta)}_1\le n^{(\be)}_0=1$ (for each $\be$), then in view of Theorem \ref{equivalencesmallsequthm} we see that $(iii)$ and $(v)$ together yield
$$\forall\;0<\beta_1<\beta_2:\;\;\;\mathcal{E}_{[N^{(\beta_1)}]}\subsetneq\mathcal{E}_{[N^{(\beta_2)}]}.$$

\item[$(*)$] In any case, $(v)$ implies that the sequences are pair-wise not equivalent.

\item[$(*)$] Finally, property $(v)$ alone is sufficient in order to have \eqref{Rmixedom1} for $\fF$.
\end{itemize}

\begin{proof}
Put $j_1:=1$ and for $k\in\NN_{>0}$ set $j_{k+1}$ to be the smallest integer $j_{k+1}>j_k$ with
\begin{equation}\label{uniformboundequ1}
(n_{j_k}^{(k)})^{1/j_k}> k (n_{j_{k+1}}^{(k+1)})^{1/j_{k+1}},
\end{equation}
and such that for all $j\ge j_{k+1}$ and all $k$ we get
\begin{equation}\label{uniformboundequ2}
\frac{(n^{(k+1)}_{j})^{1/j}}{(n^{(k)}_j)^{1/j}}\ge k.
\end{equation}
For \eqref{uniformboundequ1} we have used properties $(ii),(iii)$ and $(iv)$, and \eqref{uniformboundequ2} holds by property $(v)$.

Now put $a_0:=1$ and, for $j_k\le j<j_{k+1}$, we set
\[
(a_j)^{1/j}:=(n_{j_k}^{(k)})^{1/j_k}.
\]
Thus we have by definition that $j\mapsto(a_j)^{1/j}$ is non-increasing
and tending to $0$. \par
Finally, let $k_0 \in \N_{>0}$ be given (and from now on fixed). For $j\ge
j_{k_0+1}$ we can find $k\ge k_0$ such that $j_{k+1}\le
j<j_{k+2}$. Thus, in this situation we can estimate as follows:
\begin{align*}
    \frac{a_j^{1/j}}{(n^{(k_0)}_j)^{1/j}} = \frac{(n^{(k+1)}_{j_{k+1}})^{1/j_{k+1}}}{(n^{(k_0)}_j)^{1/j}}\ge \frac{(n^{(k+1)}_{j_{k+1}})^{1/j_{k+1}}}{(n^{(k)}_j)^{1/j}} \ge  \frac{(n^{(k+1)}_{j})^{1/j}}{(n^{(k)}_j)^{1/j}}\ge k,
\end{align*}
hence $\lim_{j\rightarrow+\infty}\frac{a_j^{1/j}}{(n^{(k_0)}_j)^{1/j}}=+\infty$. The second inequality follows from the fact that $j\mapsto(n_j^{(k+1)})^{1/j}$ is non-increasing (property $(iv)$). By the point-wise order for any $\be>0$ we can find some $k_0 \in \N_{>0}$ such that $\frac{a_j^{1/j}}{(n^{(\be)}_j)^{1/j}}\ge\frac{a_j^{1/j}}{(n^{(k_0)}_j)^{1/j}}$ for all $j\ge 1$ and hence the last desired property for $\mathbf{a}$ is verified.

Concerning \eqref{Rmixedom1}, we note that by $(v)$ we get $2^jN^{(\beta_1)}_j\le N^{(\beta_2)}_j$ for all $0<\beta_1<\beta_2$ and all $j$ sufficiently large. Consequently,
$$\forall\;0<\beta_1<\beta_2\;\exists\;C\ge 1\;\forall\;j\in\NN:\;\;\;2^jN^{(\beta_1)}_j\le CN^{(\beta_2)}_j,$$
which yields by definition of associated weights $\omega_{N^{(\beta_2)}}(2t)\le\omega_{N^{(\beta_1)}}(t)+\log(C)$ for all $t\ge 0$. This verifies \eqref{Rmixedom1} for $\fF$.
\end{proof}

\begin{remark}
The previous result shows that any family $\fF \subseteq \cL\cC$ that can be parametrized to satisfy $(ii)-(v)$  from Proposition \ref{prop:uniformbound} is already uniformly bounded by some sequence $\mathbf{a}$.

Consequently, in this case the assumptions $(i)$ and $(ii)$ from Lemma \ref{lem:Lemma4General} on the existence of $\mathbf{a}$ are superfluous and also assumption \eqref{Rmixedom1} for $\fF$ holds true automatically.
\end{remark}

Before we can give the proof of Lemma \ref{lem:Lemma4General}, we need one more technical lemma as preparation.

\begin{lemma}
\label{lem:technicallemma}
   Let $\mathbf{a} = (a_j)_j \in \R_{>0}^\N$  with  $\lim_{j\rightarrow+\infty}a_j^{1/j}=0$ be given. Then there exists a function $g = g_{\mathbf{a}}: \RR_{>0}\rightarrow \RR_{>0}$ with the following properties:
   \begin{itemize}
       \item[$(*)$] $\lim_{t\rightarrow+\infty}g_{\mathbf{a}}(t)=+\infty$.

       \item[$(*)$] For all $N\in\hyperlink{LCset}{\mathcal{LC}}$ such that $n_j \le D a_j$ (for some $D = D(N)>0$ and all $j\in\NN$), and all $d, s>0$ we have that
    \[
    \lim_{t\rightarrow+\infty}s\om_N(t/2)-dg_{\mathbf{a}}(t)t=+\oo.
    \]
   \end{itemize}
\end{lemma}

\begin{proof}
    Observe that
    \[
    \om_N(t) \ge \sup_{k \in \N} \log \frac{t^k}{Da_k k!} \ge \log\left(\frac{1}{2D}\sum_{k = 0}^{+\infty}\frac{(t/2)^k}{a_k k!} \right)=:  h_{\mathbf{a}}(t)-\log(2D).
    \]
    It is clear from the definition that $h_{\mathbf{a}}$ is non-decreasing. From the assumption $\lim_{j\rightarrow+\infty}a_j^{1/j}=0$, it follows that for every $R>0$ there exists $C\in \R$ such that for all $t>0$ we have
    \begin{equation}
    \label{eq:hdivergent}
    h_{\mathbf{a}}(t) \ge C +Rt.
    \end{equation}
    This estimate follows since for every (small) $\ve > 0$ there exists $B>0$ such that $a_k\le B \ve^k$ for all $k \in \N$; and therefore
    \[
    \log\left(\sum_{k = 0}^{+\infty}\frac{(t/2)^k}{a_k k!}\right)\ge \frac{t}{2\ve}-\log(B),
    \]
    which gives \eqref{eq:hdivergent}.\par
    Let us set $f_{\mathbf{a}}(t) := \frac{h_{\mathbf{a}}(t/2)}{t}$, then, by \eqref{eq:hdivergent}, $\lim_{t\rightarrow+\infty}f_{\mathbf{a}}(t)=+\infty$. Finally set $g_{\mathbf{a}}:=\sqrt{f_{\mathbf{a}}}$ and so $\lim_{t\rightarrow+\infty}g_{\mathbf{a}}(t)=+\infty$. Moreover, we have $\lim_{t\rightarrow+\infty}\ve f_{\mathbf{a}}(t) - g_{\mathbf{a}}(t)=+\infty$ for every $\ve >0$. Thus for any arbitrary fixed $s > 0$, we get
    \begin{equation}
    \label{eq:final}
    s\om_N(t/2) - g_{\mathbf{a}}(t)t \ge s h_{\mathbf{a}}(t/2)-s\log(2D)-g_{\mathbf{a}}(t)t=t(sf_{\mathbf{a}}(t)-g_{\mathbf{a}}(t))-s\log(2D),
    \end{equation}
    hence $\lim_{t\rightarrow+\infty}s\om_N(t/2) - g_{\mathbf{a}}(t)t=+\infty$. This shows the statement for $d=1$. For $d \neq 1$, the result simply follows by choosing $s/d$ in \eqref{eq:final}.
    \end{proof}

\begin{proof}[Proof of Lemma \ref{lem:Lemma4General}]
We adapt the proof of \cite[Lemma 4]{MarkinIII}. So assume that the operator $A$ is actually unbounded. Then the spectrum $\si(A)$ is unbounded as well and so there exists a strictly increasing sequence of natural numbers $k(n)$ such that
\begin{itemize}
    \item[$(i)$] $n \le g_{\mathbf{a}}(k(n))$ (and $n \le k(n)$) for all $n\in\NN_{>0}$,
    \item[$(ii)$] in each ring $\{\la \in \C:~ k(n) < |\la| < k(n)+1\}$ there is a point $\la_n \in \si(A)$,
\end{itemize}
    and we can actually find a $0$-sequence $\ve_n$ with $0<\ve_n<\min(1/n,\ve_{n-1})$ such that $\la_n$ belongs to the ring
    \[
    r_n := \{\la \in \C:~ k(n)-\ve_n< |\la| < k(n)+1-\ve_n\}.
    \]

    As in Markin's proof, the subspaces $E_A(r_n)H$ are non-trivial and pairwise orthogonal. Thus in each of those spaces we may choose a non-trivial element $e_n$ such that
    \[
    e_n=E_A(r_n)e_n, \quad \langle e_i, e_j\rangle = \de_{i,j}.
    \]
    Now we define
    \[
    f:=\sum_{n = 1}^{+\oo} g_{\mathbf{a}}(k(n))^{-(k(n)+1-\ve_n)}e_n.
    \]
    As in \cite{MarkinIII}, the sequence of coefficients belongs to $\ell^2$, and
    \[
    E_A(r_n)f=g_{\mathbf{a}}(k(n))^{-(k(n)+1-\ve_n)}e_n, \quad E_A(\bigcup_{n \in \N_{>0}} r_n)f=f.
    \]
    Moreover, for every $t>0$, we have
    \begin{align*}
    \int_\C e^{2t|\la|}d\langle E_A(\la)f,f\rangle &= \int_\C e^{2t|\la|}d\langle E_A(\la)E_A(\bigcup_{n \in \N_{>0}}r_n)f,E_A(\bigcup_{n \in \N_{>0}}r_n)f\rangle\\
    &=\sum_{n = 1}^\oo \int_{r_n} e^{2t|\la|}d\langle E_A(\la)f,f\rangle\\
    &=\sum_{n = 1}^\oo \int_{r_n} e^{2t|\la|}d\langle E_A(\la)E_A(r_n)f,E_A(r_n)f\rangle\\
    &=\sum_{n=1}^\oo g_{\mathbf{a}}(k(n))^{-2(k(n)+1-\ve_n)} \int_{r_n} e^{2t|\la|}d\langle E_A(\la)e_n,e_n\rangle\\
    &\le \sum_{n=1}^\oo e^{-2\log (g_{\mathbf{a}}(k(n))(k(n)+1-\ve_n)}e^{2t(k(n)+1-\ve_n)} \underbrace{\|E_A(r_n)e_n\|^2}_{=1}\\
    &= \sum_{n = 1}^{+\oo}e^{-2(\log(g_{\mathbf{a}}(k(n))-t)(k(n)+1-\ve_n)}<+\oo,
    \end{align*}
    where we used in the first inequality that for $\la \in r_n$ we have $|\la| \le k(n)+1-\ve_n$, and in the final inequality that $g_{\mathbf{a}}$ tends to infinity and that $k(n)\ge n$. Thus we have shown that $ f \in \cE_{(G^1)}(A)$.\par

    Moreover, in analogy to \cite{MarkinIII}, and by a similar reasoning as above, we get for all $N\in\fF$ and $t>0$
    \begin{align}
    \label{eq:firsteq}
        \int_\C e^{2\om_N(t|\la|)}d\langle E_A(\la)f,f\rangle= \sum_{n=1}^\oo  g_{\mathbf{a}}(k(n))^{-2(k(n)+1-\ve_n)} \int_{r_n} e^{2\om_N(t|\la|)}d\langle E_A(\la)e_n,e_n\rangle.
    \end{align}
    Next we observe that for $\la \in r_n$ we have $\om_N(t|\la|) \ge \om_N(t(k(n)-\ve_n))\ge \om_N(t(k(n)-1))$. We continue to estimate the right hand side of \eqref{eq:firsteq} and infer
    \begin{align*}
        \int_\C e^{2\om_N(t|\la|)}d\langle E_A(\la)f,f\rangle &\ge \sum_{n=1}^\oo  g_{\mathbf{a}}(k(n))^{-2(k(n)+1-\ve_n)} e^{2\om_N(t(k(n)-1))} \underbrace{\int_{r_n} d\langle E_A(\la)e_n,e_n\rangle}_{=1}\\
        &\ge \sum_{n=1}^\oo  e^{2(\om_N(t(k(n)-1))-\log(g_{\mathbf{a}}(k(n)))(k(n)+1))}.
    \end{align*}
    By iterating \eqref{Rmixedom1} there exist $M\in\fF$, $s>0 $ (small) and $C > 0$ (large) such that for all $\la\in\CC$
    \[
    \om_N(t|\la|) \ge s\om_M(|\la|)-C,
    \]
    which allows us to continue the estimate and get
    \begin{equation*}
    \label{markinIIIcontequ}
    \int_\C e^{2\om_N(t|\la|)}d\langle E_A(\la)f,f\rangle \ge \sum_{n=1}^\oo  e^{2(s\om_M((k(n)-1))-C-\log(g_{\mathbf{a}}(k(n)))(k(n)+1))}=+\oo,
    \end{equation*}
    where the last equality follows from Lemma \ref{lem:technicallemma} (applied to the sequence $M$ and $d=2$).
    Thus we infer that  $f \notin \cE_{\{N\}}(A)$. Since $N \in \fF$ has been arbitrary we are done.
\end{proof}

Finally we are now in the position to prove our main theorem, a generalization of \cite[Thm. 5.1]{MarkinIII} which reads as follows.

\begin{theorem}
\label{thm:mainMarkin}
    Suppose there exists $\mathbf{a}=(a_j)_j$ such that $\lim_{j\rightarrow+\infty}a_j^{1/j}=0$ and a family $\fF$ of weight sequences as in Lemma \ref{lem:Lemma4General}. Assume that for any weak solution $y$ of \eqref{eq:evo} on $[0,+\infty)$, there is $N \in \fF$ such that $y \in \cE_{\{N\}}([0,+\infty),H)$. Then the operator $A$ is bounded.
\end{theorem}

\begin{proof}
    Let $y$ be a weak solution of \eqref{eq:evo}. By assumption, there exists $N \in \fF$ such that $y \in \cE_{\{N\}}([0,+\infty),H)$. By Theorem \ref{thm:regsolution}, we get that for every $t \ge 0$, we have
    \[
    y(t) \in \cE_{\{N\}}(A),
    \]
    in particular $y(0) \in \cE_{\{N\}}(A)$. Via an application of \cite[Thm. 3.1]{MarkinI}, we infer
    \begin{equation}
    \label{eq:subset}
        \bigcap_{t>0}D(e^{tA})\subseteq \bigcup_{N \in \fF}\cE_{\{N\}}(A).
    \end{equation}
    On the other hand, since
    \[
    \bigcap_{t>0} D(e^{tA}) = \bigcap_{t>0} \{ f \in H:~\int_\C e^{2t \cR(\la)} \langle d E_A(\la)f,f \rangle <+\oo \},
    \]
    it is clear that
    \[
    \bigcap_{t>0} D(e^{tA}) \supseteq \bigcap_{t>0} \{ f \in H:~\int_\C e^{2t |\la|} \langle d E_A(\la)f,f \rangle <+\oo \} = \cE_{(G^1)}(A).
    \]
    Together with \eqref{eq:subset} this yields
    \[
    \bigcup_{N \in \fF}\cE_{\{N\}}(A) = \cE_{(G^1)}(A).
    \]
    Thus, by using Lemma \ref{lem:Lemma4General} we conclude that $A$ is bounded.
\end{proof}

When taking $\fF$ to be the family of all small Gevrey sequences, i.e. $\fF=\fG:=\{G^\al:~\al < 1\}$, we infer \cite[Thm. 5.1]{MarkinIII} (see also Remark \ref{rem:Markin}).

\subsection{An answer to the motivating question from Section \ref{sec:motivating}}\label{answersection}
The final goal is now to combine the information from Theorem \ref{thm:main} and Theorem \ref{thm:mainMarkin}.

So, suppose $\fF$ is a family of weight sequences such that:

\begin{itemize}
\item[$(i)$] $N\in\hyperlink{LCset}{\mathcal{LC}}$ for all $N\in\fF$ and $1=N_0=N_1$,

\item[$(ii)$]  $\fF$ has \eqref{Rmixedom1},

\item[$(iii)$] $\fF$ is uniformly bounded by some $\mathbf{a}=(a_j)_j$ with $\lim_{j\rightarrow+\infty}a_j^{1/j}=0$, and


\item[$(iv)$] for all $N\in\fF$ we have that $n$ is log-concave.
\end{itemize}

Note that $(iii)$ gives $\lim_{j\rightarrow+\infty}(n_j)^{1/j}=0$ for all $N \in \fF$. So $\fF$ is a family as required in Lemma \ref{lem:Lemma4General} and by $(i)$, $(iii)$ and $(iv)$ Theorem \ref{thm:main} can be applied to each $N\in\fF$, hence
$$\forall\;N\in\fF:\;\;\;\cE_{\{N\}}(I,H) \cong \cH^\oo_{\underline{\cN^*}_\fc}(\C,H).$$
Summarizing, we can reformulate Theorem \ref{thm:mainMarkin} as follows.

\begin{theorem}\label{finalmarkinthm}
Let $\fF$ be a family of weight sequences as considered before. Suppose that for every weak solution $y$ of \eqref{eq:evo} there exist $N \in \fF$ and $C,k>0$ such that $y$ can be extended to an entire function with
    \[
    \|y(z)\|\le Ce^{\om_{N^*}(k|z|)}.
    \]
    Then $A$ is already a bounded operator.
\end{theorem}

Theorem \ref{finalmarkinthm} applies to the family $\fG:=\{G^{\alpha}: 0\le\alpha<1\}$ of all small Gevrey sequences.

\appendix
\section{On dual weight sequences and Matuszewska indices}\label{dualsection}

The growth and regularity assumptions for weight sequences $M$ in Theorem \ref{thm:main} or for $N\in\fF$ in Lemma \ref{lem:Lemma4General}, in the technical Proposition \ref{prop:uniformbound} and in Theorems \ref{thm:mainMarkin}, \ref{finalmarkinthm} are by far not standard in the theory of ultradifferentiable (and ultraholomorphic) functions. More precisely the sequences under consideration are required to grow very slowly or to be even non-increasing. This is due to the fact that in Theorem \ref{thm:main} resp. in Theorem \ref{finalmarkinthm} the {\itshape conjugate sequence} $M^{*}$ resp. $N^{*}$ plays the crucial role in order to restrict the growth. Therefore the conjugate sequence(s) is (are) required to satisfy the frequently used conditions in the weight sequence setting; e.g. in order to work with the associated function $\omega_{M^{*}}$. \par

We are interested in studying and constructing such ``exotic/non-standard'' sequences and may ask how they are ``naturally'' related to standard sequences. On the one hand, as already stated in Section \ref{conjugatesect}, formally we can start with a standard/regular sequence $R=M^{*}$ and then get $M$ by the formula \eqref{Mstardef} which relates $M$ and $M^{*}$ by a one-to-one correspondence; i.e. take $M=R^{*}$. However, in this Section the aim is to give a completely different approach and to show how such ``exotic'' small sequences $M$ are appearing and can be introduced in a natural way. The main idea is to start with $N\in\hyperlink{LCset}{\mathcal{LC}}$ (and having some more standard requirements) and then consider the so-called {\itshape dual sequence} $D$ from \cite[Sect. 2.1.5]{dissertationjimenez}. However, in order to proceed we also have to recall and study the notion of {\itshape Matuszewska indices.}

\subsection{Matuszewska indices}\label{Matuszeskasection}
We recall some facts and definitions from \cite[Sect. 2.1.2]{dissertationjimenez}, see also the literature citations therein and especially \cite{regularvariation}. Moreover we refer to \cite[Sect. 3]{index}. Note that in \cite{dissertationjimenez} and in \cite{index} a sequence $M\in\RR_{>0}^{\NN}$ is called a weight sequence if it satisfies all requirements from the class $\hyperlink{LCset}{\mathcal{LC}}$ except necessarily $M_0\le M_1$, see \cite[Sect. 1.1.1, p. 29; Def. 1.1.8, p.32]{dissertationjimenez} and \cite[Sect. 2.2, Sect. 3.1]{index}.\vspace{6pt}

First, for any given sequence $\mathbf{a}=(a_p)_p\in\RR_{>0}^{\NN}$ the {\itshape upper Matuszewska index} $\alpha(\mathbf{a})$ is defined by
\begin{align*}
    \alpha(\mathbf{a}):=&\inf\{\alpha\in\RR: \frac{a_p}{p^{\alpha}}\;\text{is almost decreasing}\}\\
    =&\inf\{\alpha\in\RR: \exists\;H\ge 1\;\forall\;1\le p\le q:\;\;\;\frac{a_q}{q^{\alpha}}\le H\frac{a_p}{p^{\alpha}}\},
\end{align*}
and the {\itshape lower Matuszewska index} $\beta(\mathbf{a})$ by
\begin{align*}
    \beta(\mathbf{a}):=&\sup\{\beta\in\RR: \frac{a_p}{p^{\beta}}\;\text{is almost increasing}\}\\
    =&\sup\{\beta\in\RR: \exists\;H\ge 1\;\forall\;1\le p\le q:\;\;\;\frac{a_p}{p^{\beta}}\le H\frac{a_q}{q^{\beta}}\}.
\end{align*}
Note that $\beta(\mathbf{a})>0$ implies, in particular, $\lim_{p\rightarrow+\infty}a_p=+\infty$.\vspace{6pt}

The aim is to give a connection between these indices and the notion of the conjugate sequence introduced in this work. The following comments $(a)-(e)$ and Lemma \ref{betaindexlemma} have been made resp. suggested by the anonymous referee:

First put $\mathbf{g}^1:=(p)_{p\in\NN}$ and $\mathbf{a}^{-1}:=(a_p^{-1})_{p\in\NN}$. Consequently, by definition of the above indices the following relations are valid, see also \cite[Rem. 2.6, Prop. 3.6]{index} applied to $r=1$ and $s=-1$:
\begin{equation}\label{firstnewid}
\alpha(\mathbf{g}^1\mathbf{a}^{-1})=1+\alpha(\mathbf{a}^{-1})=1-\beta(\mathbf{a}),
\end{equation}
and
\begin{equation}\label{firstnewid1}
\beta(\mathbf{g}^1\mathbf{a}^{-1})=1+\beta(\mathbf{a}^{-1})=1-\alpha(\mathbf{a}).
\end{equation}
The idea is now to apply these identities to $\mathbf{a}\equiv\mu$ and so $\mathbf{g}^1\mathbf{a}^{-1}$ corresponds to $\mu^{*}$, i.e. the sequence of quotients of the conjugate sequence $M^{*}$ (recall \eqref{Mstardef}, \eqref{Mstardefquot}). Combining this information with results from \cite{index} we summarize:

\begin{itemize}
\item[$(a)$] By \eqref{firstnewid} one has $\alpha(\mathbf{g}^1\mathbf{a}^{-1})\le 1$ if and only if $\beta(\mathbf{a})\ge 0$ resp. with strict inequalities. In particular, if $M$ is log-convex, then $\beta(\mu)\ge 0$ and so $\alpha(\mathbf{\mu}^{*})\le 1$. Conversely, if $\alpha(\mathbf{\mu}^{*})<1$ and so $\beta(\mu)>0$, then $M$ is equivalent to a log-convex sequence $L$ and more precisely the equivalence is even established on the level of quotient sequences (see the proof of \cite[Prop. 4.15]{JimenezGarridoSanz} and $(a)$ in Remark \ref{almoststuff} for the analogous estimates in \eqref{almoststuffequ1}).

    This should be compared with \cite[Thm. 3.16, Cor. 3.17]{index} applied to $M^{*}$ resp. $L^{*}$ and $(v)$ in Section \ref{sec:propertiesconjugate}. (Since $L^{*}$ is equivalent to $M^{*}$ also $M^{*}$ has \hyperlink{mg}{$(\on{mg})$}.)

Indeed, if $\beta(\mu)>0$, then $M$ is equivalent to a sequence $L\in\hyperlink{LCset}{\mathcal{LC}}$ because $\lim_{p\rightarrow+\infty}\mu_p=\lim_{p\rightarrow+\infty}\lambda_p=+\infty$ and so one can achieve normalization by changing finitely many terms of $L$ at the beginning; see $(iv)$ in Remark \ref{dualremark} and also the proof in Lemma \ref{dual1}.\vspace{6pt}

 Similarly, $\beta(\mathbf{g}^1\mathbf{a}^{-1})\ge 0$ if and only if $\alpha(\mathbf{a})\le 1$ holds by \eqref{firstnewid1} resp. with strict inequalities and the above comments apply when $M$ is replaced by $M^{*}$ and $M^{*}$ by $M$.

 \item[$(b)$] Using this knowledge, we can change the assumptions in Theorem \ref{thm:main} as follows: In order to proceed we take $M \in \R^\N_{>0}$ such that $\alpha(\mu)<1$ is valid. Because then $\beta(\mu^{*})>0$, hence $M^{*}$ is equivalent to a sequence $L^{*}\in\hyperlink{LCset}{\mathcal{LC}}$ and this property is sufficient to proceed by taking into account $(ii)$ in Remark \ref{mainthmremark} and $(ii)$ in Section \ref{sec:propertiesconjugate}.

     The same comment applies to Theorem \ref{equivalencesmallsequthm}; i.e. we are taking $M,N \in \R^\N_{>0}$ with $\alpha(\mu),\alpha(\nu)<1$.\vspace{6pt}

     In order to ensure this requirement we give several ideas: First, when given $M$, in Theorem \ref{dual14} we deal with the corresponding dual sequence $D$ and so we want to have $\alpha(\delta)<1$. This can be expressed in terms of $M$; see the next Section \ref{dualsequencesection} for details.

     However, even directly for $M$ we can get $\alpha(\mu)<1$; in this context see also Theorem \ref{dualdual}: For this let $M\in\RR_{>0}^{\NN}$ be given and assume that either $\beta(\nu)>1$ or that $\alpha(\mu)<+\infty$. In the first case we take $M^{-1}:=(M^{-1}_p)_{p\in\NN}$ and in the second one $M$ directly if already $\alpha(\mu)<1$ resp. $G^{-r+1}M:=(p!^{-r+1}M_p)_{p\in\NN}$ if $1\le\alpha(\mu)<r$.

 \item[$(c)$] In $(i)$ in Lemma \ref{lem:weightedentireiso} we assume $M\in\RR_{>0}^{\NN}$ such that $0<\beta(\mu)\le\alpha(\mu)<+\infty$: First, $0<\beta(\mu)$ yields that $M$ is equivalent to $L\in\hyperlink{LCset}{\mathcal{LC}}$ and since the equivalence is established on the level of the quotient sequences (see $(a)$ above) we get $0<\beta(\lambda)\le\alpha(\lambda)<+\infty$, too.

     By \cite[Thm. 3.16, Cor. 3.17]{index} property $\alpha(\lambda)<+\infty$ is equivalent to having \hyperlink{mg}{$(\on{mg})$} for $L$. Moreover, by combining \cite[Prop. 3.4]{subaddlike} and \cite[Thm. 3.11]{index} applied to the sequence $L$ and $\beta=0$ the second assumption $0<\beta(\lambda)$ yields \eqref{eq:Momega1} for $L$. Since the weighted classes appearing in $(ii)$, $(iii)$ in Lemma \ref{lem:weightedentireiso} are preserved under equivalence of weight sequences (recall $(ii)$ in Remark \ref{mainthmremark}) we are done. (Note that $M$ has both \hyperlink{mg}{$(\on{mg})$} and \eqref{eq:Momega1} too since equivalence preserves these requirements; for \hyperlink{mg}{$(\on{mg})$} this is clear and concerning \eqref{eq:Momega1} see the proof of \cite[Cor. 3.3]{subaddlike}.)

 \item[$(d)$] In Corollary \ref{cor:weightedentireiso}, first the aim is to apply Lemma \ref{lem:weightedentireiso} to $M^{*}$ and therefore we assume that $M\in\RR_{>0}^{\NN}$ is given such that $0<\beta(\mu^{*})\le\alpha(\mu^{*})<+\infty$ which is equivalent to requiring $-\infty<\beta(\mu)\le\alpha(\mu)<1$ (see $(a)$). The first estimate is clearly true if $M$ is log-convex and by $(b)$ the second one is sufficient to apply Theorem \ref{thm:main} which is needed in the proof of Corollary \ref{cor:weightedentireiso} as well.

 \item[$(e)$] Finally, for the sake of completeness we comment on the meaning of \eqref{om1} and \eqref{eq:Momega1} in terms of growth indices: First, let us consider the auxiliary sequence $R\in\RR_{>0}^{\NN}$ defined via the quotients $\rho=(\rho_p)_p$ with $\rho_p:=(M_p)^{1/p}$, $p\ge 1$. Hence $R_p=\prod_{i=1}^p(M_i)^{1/i}$ with $R_0=1$ (empty product) and so $M\in\hyperlink{LCset}{\mathcal{LC}}$ implies $R\in\hyperlink{LCset}{\mathcal{LC}}$. Then \cite[Thm. 3.11]{index} applied to $R$ and $\beta=0$ yields that \eqref{eq:Momega1} for $M\in\hyperlink{LCset}{\mathcal{LC}}$ is equivalent to $\beta(\rho)>0$.

     On the other hand recall that \eqref{om1} means $\alpha(\omega_M)<+\infty$; i.e. \cite[Thm. 2.11, Cor. 2.14]{index} applied to $\sigma=\omega_M$ with $\alpha$ denoting the index for functions from \cite[Sect. 2.2]{index}.\vspace{6pt}

      Summarizing, \cite[Thm. 3.1]{subaddlike} precisely shows that for any $M\in\hyperlink{LCset}{\mathcal{LC}}$ we have $\alpha(\omega_M)<+\infty$ if and only if $\beta(\rho)>0$.\vspace{6pt}

      If $M$ has in addition \hyperlink{mg}{$(\on{mg})$} (e.g. like in $(c)$ above), then $\beta(\rho)=\beta(\mu)$ and $\alpha(\rho)=\alpha(\mu)$ holds true: Both equalities follow by the estimates $\rho_p=(M_p)^{1/p}\le\mu_p\le A\rho_p$ for some constant $A\ge 1$ and all $p\ge 1$; recall \eqref{mucompare1} for the first and e.g. \cite[Lemma 3.1 $(iii)$]{index} for the second one.
\end{itemize}

Based on comment $(e)$ the following question appears: What can be said about the relation between $\beta(\rho)$ and $\beta(\mu)$ in general; i.e. when $M$ does not have \hyperlink{mg}{$(\on{mg})$}. In order to prove relation \eqref{betaindexcomparison}, which has been claimed by the referee, for technical reasons we have to recall some more notation also used in Section \ref{dualsequencesection} below:

Let $M\in\hyperlink{LCset}{\mathcal{LC}}$ be given. We introduce the {\itshape counting function}
\begin{equation*}\label{sigmam}
\Sigma_M(t):=|\{p\in\NN_{>0}: \mu_p\le t\}|,\;\;\;t\ge 0.
\end{equation*}
By definition it is obvious that $\Sigma_M(t)=0$ on $[0,\mu_1)$ and $\Sigma_M(t)=p$ on $[\mu_p,\mu_{p+1})$ provided that $\mu_p<\mu_{p+1}$. Recall that for $M\in\hyperlink{LCset}{\mathcal{LC}}$ we have $\lim_{p\rightarrow+\infty}\mu_p=+\infty$. Moreover, we recall the known integral representation formula (see \cite[1.8. III]{mandelbrojtbook} and also \cite[$(3.11)$]{Komatsu73}):
\begin{equation}\label{assointrepr}
\omega_M(t)=\int_0^t\frac{\Sigma_M(u)}{u}du=\int_{\mu_1}^t\frac{\Sigma_M(u)}{u}du.
\end{equation}

\begin{lemma}\label{betaindexlemma}
Let $M\in\hyperlink{LCset}{\mathcal{LC}}$ be given and let $R$ be the sequence given by $R_p=\prod_{i=1}^p(M_i)^{1/i}$, $p\in\NN$. Then we get
\begin{equation}\label{betaindexcomparison}
\beta(\rho)\ge\beta(\mu)\ge 0.
\end{equation}
\end{lemma}

\begin{proof}
The arguments and techniques are based on the proofs of the characterization \cite[Thm. 3.1]{subaddlike} and of \cite[Lemma 12, $(2)\Rightarrow(4)$]{BonetMeiseMelikhov07} and the obtained estimates might have applications in different contexts as well.\vspace{6pt}

First, recall that $\beta(\mu)\ge 0$ because $M$ is assumed to be log-convex. If $\beta(\mu)=0$, then \eqref{betaindexcomparison} is trivial since $R$ is log-convex too and hence $\beta(\rho)\ge 0$. On the other hand, if $\beta(\rho)=0$ then $\beta(\rho)=\beta(\mu)$ has to be valid and so \eqref{betaindexcomparison} is clear, too. So let from now on $\beta(\mu)>0$ and $\beta(\rho)>0$.\vspace{6pt}

We take $0\le\beta<\beta(\mu)$ and hence \cite[Thm. 3.11, $(v)\Leftrightarrow(vii)$]{index} gives
\begin{equation}\label{betaindexlemmaequ0}
\exists\;k\in\NN_{\ge 2}:\;\;\;\liminf_{p\rightarrow+\infty}\frac{\mu_{kp}}{\mu_p}>k^{\beta},
\end{equation}
hence $\mu_{kp}>k^{\beta}\mu_p$ holds for all $p\ge p_{\beta,k}$. Then let $t\ge\mu_{p_{\beta,k}}$ and so $\mu_p\le t<\mu_{p+1}$ for some $p\ge p_{\beta,k}$. We get by the definition of the counting function $\Sigma_M(t)=p$ and also  $\Sigma_M(k^{\beta}t)\le\Sigma_M(k^{\beta}\mu_{p+1})<k(p+1)=k\Sigma_M(t)+k$ follows. Consequently, so far we have shown that
\begin{equation}\label{betaindexlemmaequ}
\forall\;0\le\beta<\beta(\mu)\;\exists\;k\in\NN_{\ge 2}\;\exists\;D\ge 1\;\forall\;t\ge 0:\;\;\;\Sigma_M(k^{\beta}t)\le k\Sigma_M(t)+D.
\end{equation}
Using \eqref{betaindexlemmaequ} and the integral representation \eqref{assointrepr} we estimate for all $t\ge\mu_1/k^{\beta}$ as follows:
\begin{align*}
\omega_M(k^{\beta}t)&=\int_{\mu_1}^{k^{\beta}t}\frac{\Sigma_M(u)}{u}du=\int_{\mu_1/k^{\beta}}^t\frac{\Sigma_M(k^{\beta}v)}{v}dv
\\&
\le k\int_0^t\frac{\Sigma_M(v)}{v}dv+D\int_{\mu_1/k^{\beta}}^t\frac{1}{v}dv=k\omega_M(t)+D\log(tk^{\beta}/\mu_1).
\end{align*}
Since $\omega_M(t)=o(\log(t))$ as $t\rightarrow+\infty$ we have shown now
\begin{equation}\label{betaindexlemmaequ1}
\forall\;0\le\beta<\beta(\mu)\;\exists\;k\in\NN_{\ge 2}\;\exists\;D_1\ge 1\;\forall\;t\ge 0:\;\;\;\omega_M(k^{\beta}t)\le(k+1)\omega_M(t)+D_1.
\end{equation}
Then note that $k+1\le 2k$ for all $k\in\NN_{\ge 2}$ and, when \eqref{betaindexlemmaequ0} is valid for some $k$, then also for all $k^i$ because by iteration we get $\frac{\mu_{k^ip}}{\mu_p}>k^{i\beta}$. Thus, when $0\le\beta'<\beta(\mu)$ is given we can choose $\beta$ with $\beta'<\beta<\beta(\mu)$ and $k$ sufficiently large to ensure $(2k)^{\beta'}\le k^{\beta}\Leftrightarrow 2^{\beta'}\le k^{\beta-\beta'}$. Therefore, when taking $k_1:=2k$ finally we arrive at
\begin{equation}\label{betaindexlemmaequ2}
\forall\;0\le\beta'<\beta(\mu)\;\exists\;k_1\in\NN_{\ge 2}\;\exists\;B\ge 1\;\forall\;t\ge 0:\;\;\;\omega_M(k_1^{\beta'}t)\le k_1\omega_M(t)+B.
\end{equation}

Now we move to the study of $\beta(\rho)$: Let $0\le\beta<\beta(\rho)$ and so \cite[Thm. 3.11, $(v)\Leftrightarrow(vii)$]{index} applied to $R$ gives
$$\exists\;k\in\NN_{\ge 2}:\;\;\;\liminf_{p\rightarrow+\infty}\frac{(M_{kp})^{1/(kp)}}{(M_p)^{1/p}}>k^{\beta}.$$
Then we replace in \cite[Thm. 3.1, $(i)\Rightarrow(ii)$]{subaddlike} the value $h$ by $k^{\beta}$ and $L$ by $k$ and get
\begin{equation}\label{betaindexlemmaequ3}
\forall\;0\le\beta<\beta(\rho)\;\exists\;k\in\NN_{\ge 2}\;\exists\;B\ge 1\;\forall\;t\ge 0:\;\;\;\omega_M(k^{\beta}t)\le 2k\omega_M(t)+B.
\end{equation}
Analogously as before this implies
\begin{equation}\label{betaindexlemmaequ4}
\forall\;0\le\beta'<\beta(\rho)\;\exists\;k_1\in\NN_{\ge 2}\;\exists\;B\ge 1\;\forall\;t\ge 0:\;\;\;\omega_M(k_1^{\beta'}t)\le k_1\omega_M(t)+B.
\end{equation}
Conversely, using \eqref{betaindexlemmaequ4} and following \cite[Thm. 3.1, $(ii)\Rightarrow(i)$]{subaddlike} and replacing there $h$ by $k_1^{\beta'}$ and $L'$ by $k_1$ then we get (with the same choice $k_1$ for given $\beta'$) the estimate
\begin{equation*}\label{betaindexlemmaequ5}
\forall\;0\le\beta'<\beta(\rho)\;\exists\;k_1\in\NN_{\ge 2}:\;\;\;\liminf_{p\rightarrow+\infty}\frac{(M_{k_1p})^{1/(k_1p)}}{(M_p)^{1/p}}>k_1^{\beta'}.
\end{equation*}
Therefore, in order to verify $0\le\beta<\beta(\rho)$, equivalently one can use \eqref{betaindexlemmaequ4} and by comparing this with \eqref{betaindexlemmaequ2} above we have shown $\beta(\rho)\ge\beta(\mu)$.
\end{proof}

\subsection{Dual sequences}\label{dualsequencesection}
Let $N\in\hyperlink{LCset}{\mathcal{LC}}$ be given. We define a new sequence $D$, called its {\itshape dual sequence}, in terms of its quotients $\delta=(\delta_p)_{p\in\NN}$ as follows, see \cite[Def. 2.1.40, p. 81]{dissertationjimenez}:
\begin{equation*}\label{dualdef}
\forall\;p\ge\nu_1(\ge 1):\;\;\;\delta_{p+1}:=\Sigma_N(p),\hspace{20pt}\delta_{p+1}:=1\;\;\;\forall\;p\in\Z,\;-1\le p<\nu_1,
\end{equation*}
and set $D_p:=\prod_{i=0}^p\delta_i$. Hence $D\in\hyperlink{LCset}{\mathcal{LC}}$ with $1=D_0=D_1$ follows by definition.

Please note that in \cite{dissertationjimenez} and \cite{index} a different notation for the counting function and the sequence of quotients of a weight sequence has been used and that concerning the definition of the sequence of quotients an index shift appears; see \cite[Def. 1.1.2, Def. 2.1.27]{dissertationjimenez} for details.\vspace{6pt}


In \cite[Thm. 2.1.43, p. 82]{dissertationjimenez} the following result has been shown:

\begin{theorem}\label{Javithm2143}
Let $N\in\hyperlink{LCset}{\mathcal{LC}}$ be given and such that
\begin{equation}\label{betweenmgdc}
\exists\;A\ge 1\;\forall\;p\in\NN:\;\;\;\nu_{p+1}\le A\nu_p.
\end{equation}
Then we get $\alpha(\nu)=\frac{1}{\beta(\delta)}$ and $\beta(\nu)=\frac{1}{\alpha(\delta)}$.
\end{theorem}

{\itshape Note:}

\begin{itemize}
\item[$(i)$] As pointed out in \cite[Sect. 2.1.3, p. 63-64]{dissertationjimenez} and \cite[Remark 3.8]{index}, the aforementioned index shift in the sequences of quotients is not effecting the value of the Matuszewska indices $\alpha(\cdot)$ and $\beta(\cdot)$.

\item[$(ii)$] \eqref{betweenmgdc}, see \cite[(2.11), p. 76]{dissertationjimenez} and which has also appeared due to technical reasons in \cite{whitneyextensionweightmatrix}, is connected to the growth behaviors {\itshape moderate growth} and {\itshape derivation closedness}. More precisely, in \cite[Remark 2.1.36, p. 78]{dissertationjimenez} it has been shown that for log-convex sequences we have
\begin{equation}\label{mgstrangedc}
\hyperlink{mg}{(\on{mg})}\Longrightarrow\eqref{betweenmgdc}\Longrightarrow\hyperlink{dc}{(\on{dc})},
\end{equation}
and each implication cannot be reversed in general.
\end{itemize}

\subsection{Main statements}\label{maindualsection}

First, by applying Theorem \ref{Javithm2143} we immediately get the following statement.

\begin{lemma}\label{dual1}
Let $N\in\hyperlink{LCset}{\mathcal{LC}}$ be given with \eqref{betweenmgdc}. Assume that $N$ satisfies
\begin{equation}\label{dual1equ1}
\exists\;H\ge 1\;\exists\;\beta>1\;\forall\;1\le p\le q:\;\;\;\frac{\nu_p}{p^{\beta}}\le H\frac{\nu_q}{q^{\beta}},
\end{equation}
i.e. the sequence $(\nu_p/p^{\beta})_p$ is almost increasing for some $\beta>1$.

Then the dual sequence $D$ is equivalent to a sequence $L$ such that $L^{*}$ is normalized and log-convex (and $D^{*}$ is equivalent to $L^{*}$, too).
\end{lemma}

\begin{proof}
By assumption we have $\beta(\nu)\ge\beta>1$ and so $\alpha(\delta)<1$ follows by Theorem \ref{Javithm2143}. Consequently, we have that
$$\exists\;H\ge 1\;\forall\;1\le p\le q:\;\;\;\frac{\delta_q}{q}\le H\frac{\delta_p}{p},$$
i.e. $p\mapsto\frac{\delta_p}{p}$ is almost decreasing. If we can choose $H=1$, then we are done with $L\equiv D$ since $d:=(D_p/p!)_{p\in\NN}$ directly is log-concave and so $D^{*}$ is log-convex, see $(iv)$ in Section \ref{sec:propertiesconjugate} and $(a)$ in Lemma \ref{MstarLC}. Note that $D_0=D_1=1$ by definition and so $D^{*}$ is normalized, too.

If $H>1$, then we are applying $(a)$ in Remark \ref{almoststuff} to $M\equiv D$ in order to switch from $D$ to the equivalent sequence $L$ defined via \eqref{almoststuffequ}. Thus $p\mapsto\frac{\lambda_p}{p}$ is non-increasing and hence $l:=(L_p/p!)_{p\in\NN}$ is log-concave which is equivalent to the log-convexity for $L^{*}$. Normalization for $L^{*}$ follows since $D_0=D_1=1$ and finally $D^{*}$ is equivalent to $L^{*}$ which holds by $(ii)$ in Section \ref{sec:propertiesconjugate}.
\end{proof}

\begin{lemma}\label{dual2}
Let $N\in\hyperlink{LCset}{\mathcal{LC}}$ be given with $\lim_{p\rightarrow+\infty}(n_p)^{1/p}=+\infty$. Then we get $\lim_{p\rightarrow+\infty}\delta_p/p=\lim_{p\rightarrow+\infty}(d_p)^{1/p}=0$.
\end{lemma}

\begin{proof}
First, by \eqref{mucompare1} and Stirling's formula we see that $\lim_{p\rightarrow+\infty}(n_p)^{1/p}=+\infty$ implies $\lim_{p\rightarrow+\infty}\nu_p/p=+\infty$ as well.

Let $C\ge 1$ be given, arbitrary but from now on fixed. Then we can find some $p_C\in\NN_{>0}$ such that $\nu_p>pC$ for all $p\ge p_C$ holds true. Since $\lfloor\frac{p}{C}\rfloor\ge\frac{p}{C}-1\ge p_C$ is valid for all $p\in\NN$ with $p\ge Cp_C+C(>p_C)$ we have for all such (large) integers $p$ that
$$\nu_{\lfloor p/C\rfloor}>\lfloor p/C\rfloor C\ge\left(\frac{p}{C}-1\right)C=p-C\ge\frac{p}{2},$$
where the last estimate is equivalent to having $p\ge 2C$ which holds true since $p\ge Cp_C+C\ge C+C=2C$. Consequently, by the definition of the counting function $\Sigma_N$ and the dual sequence we have shown $\Sigma_N(p/2)<\lfloor\frac{p}{C}\rfloor\le\frac{p}{C}$ and so $\delta_{p+1}=\Sigma_N(p)<\frac{2p}{C}$ for all sufficiently large integers $p$. Now, since $C$ can be taken arbitrary large it follows that $\lim_{p\rightarrow+\infty}\delta_p/p=0$.

Finally, since $D\in\hyperlink{LCset}{\mathcal{LC}}$ by \eqref{mucompare} and Stirling's formula we see that $\lim_{p\rightarrow+\infty}\delta_p/p=0$ does imply $\lim_{p\rightarrow+\infty}(d_p)^{1/p}=0$.

\end{proof}

Consequently, when combining Lemmas \ref{dual1} and \ref{dual2} we have that the sequence $L$ defined via \eqref{almoststuffequ} and being equivalent to $D$ has $\lim_{p\rightarrow+\infty}\lambda_p/p=\lim_{p\rightarrow+\infty}(l_p)^{1/p}=0$, too.\vspace{6pt}

Concerning these Lemmas we comment:

\begin{remark}\label{dualremark}
Let $N$ satisfy the assumptions from Lemmas \ref{dual1} and \ref{dual2}. Then we get for the technical sequence $L$ constructed via the dual sequence $D$ the following (see again $(a)$ in Remark \ref{almoststuff} applied to $D$):

\begin{itemize}
\item[$(i)$] $L^{*}\in\hyperlink{LCset}{\mathcal{LC}}$ is valid.

\item[$(ii)$] Since $D$ is log-convex and equivalence between sequences preserves \hyperlink{mg}{$(\on{mg})$}, by $(v)$ in Section \ref{sec:propertiesconjugate} we have that both $D^{*}$ and $L^{*}$ have \hyperlink{mg}{$(\on{mg})$}.

\item[$(iii)$] Moreover, log-convexity for $D$ implies this property for $L$ and, indeed, $L$ satisfies all requirements of sequences belonging to the class \hyperlink{LCset}{$\mathcal{LC}$} except $L_0\le L_1$ because only $\lambda_1\le\delta_1=1$ is known (see \eqref{almoststuffequ1}).

\item[$(iv)$] However, when technically modifying $L$ at the beginning with the following trick one can achieve w.l.o.g. that even $L\in\hyperlink{LCset}{\mathcal{LC}}$:

When $\lambda_1=1$, then no modification is required. So let now $\lambda_1<1$. Since $L$ is log-convex the mapping $p\mapsto\lambda_p$ is non-decreasing and $\lim_{p\rightarrow+\infty}\lambda_p=+\infty$ because $L$ is equivalent to $D$. Thus there exists $p_0\in\NN_{>0}$ (chosen minimal) such that for all $p>p_0$ we have $\lambda_p\ge 1$. Then replace $L$ by $\widetilde{L}$ defined in terms of its quotients $\widetilde{\lambda}_p$, i.e. putting $\widetilde{L}_p=\prod_{i=0}^p\widetilde{\lambda}_i$, where we set
$$\widetilde{\lambda}_p:=1,\;\;\;\text{for}\;0\le p\le p_0,\hspace{15pt}\widetilde{\lambda}_p:=\lambda_p,\;\;\;\text{for}\;p>p_0.$$
Consequently we get: $1=\widetilde{L}_0=\widetilde{L}_1$, $\widetilde{L}$ is log-convex since $p\mapsto\widetilde{\lambda}_p$ is non-decreasing and $L\le\widetilde{L}\le cL$ for some $c\ge 1$ which yields that $\widetilde{L}$ and $L$ are equivalent.

Finally, $\widetilde{l}$ is log-concave since $p\mapsto\frac{\widetilde{\lambda}_p}{p}$ is non-increasing which can be seen as follows: Clearly, $\frac{\widetilde{\lambda}_p}{p}\ge\frac{\widetilde{\lambda}_{p+1}}{p+1}$ for all $1\le p\le p_0-1$ and also for all $p>p_0$ since $l$ is log-concave. Then note that $\frac{1}{p}\ge\frac{\lambda_p}{p}$ for all $1\le p\le p_0$ and so $\frac{\widetilde{\lambda}_{p_0}}{p_0}=\frac{1}{p_0}\ge\frac{\lambda_{p_0}}{p_0}\ge\frac{\lambda_{p_0+1}}{p_0+1}=\frac{\widetilde{\lambda}_{p_0+1}}{p_0+1}$.\vspace{6pt}

Summarizing (see $(a)$ in Remark \ref{almoststuff}) we have that $\widetilde{L},\widetilde{L}^{*}\in\hyperlink{LCset}{\mathcal{LC}}$, $\widetilde{L}$ is equivalent to $D$ and $\widetilde{L}^{*}$ is equivalent to $D^{*}$.
\end{itemize}
\end{remark}

\begin{remark}\label{dualremark1}
By the characterization given in \cite[Thm. 3.11]{index} and \cite[Thm. 3.10]{index}, see also \cite[Prop. 2.1.22, p. 68]{dissertationjimenez} and the discussion after the proof of \cite[Thm. 3.11]{index}, we have the following:

\eqref{dual1equ1}, i.e. $\beta(\nu)>1$, is equivalent to the fact that $N\in\hyperlink{LCset}{\mathcal{LC}}$ has \hyperlink{gamma1}{$(\gamma_1)$} or equivalently \hyperlink{beta1}{$(\beta_1)$}.

Thus $\beta(\nu)>1$ if and only if $N$ is {\itshape strongly non-quasianalytic}.

Recall that \hyperlink{gamma1}{$(\gamma_1)$} for $N$ implies, in particular, that $\lim_{p\rightarrow+\infty}(n_p)^{1/p}=+\infty$.
\end{remark}




Summarizing everything, in particular the information from Lemmas \ref{dual1} and \ref{dual2} and Remark \ref{dualremark}, we get the following main result.

\begin{theorem}\label{dual14}
Let $N\in\hyperlink{LCset}{\mathcal{LC}}$ be given and let $D\in\hyperlink{LCset}{\mathcal{LC}}$ denote the corresponding dual sequence. We assume that:

\begin{itemize}
\item[$(*)$] $\beta(\nu)>1$ holds true, i.e. $N$ is strongly non-quasianalytic and hence $\lim_{p\rightarrow+\infty}(n_p)^{1/p}=+\infty$, and

\item[$(*)$] $N$ satisfies \eqref{betweenmgdc}.
\end{itemize}

Then there exists $L\in\RR_{>0}^{\NN}$ (given by \eqref{almoststuffequ} w.r.t. the sequence $D$) which is equivalent to $D$ and such that $L$ satisfies all requirements in order to apply Theorem \ref{thm:main} to $L$. Moreover, the corresponding isomorphisms are valid for the class defined by $D$ as well (see $(ii)$ in Remark \ref{mainthmremark}) and we also have $\alpha(\delta)=\alpha(\lambda)<1$. Finally, $L$ is log-convex, $D^{*}$ and $L^{*}$ are equivalent and both satisfy \hyperlink{mg}{$(\on{mg})$}.
\end{theorem}

\begin{proof}
This follows directly by involving Lemmas \ref{dual1} and \ref{dual2}, Remark \ref{dualremark} and the comments listed in Section \ref{sec:propertiesconjugate}.
\end{proof}

\begin{corollary}\label{dual14cor}
Let $N\in\RR_{>0}^{\NN}$ satisfy the following conditions:
\begin{itemize}
\item[$(*)$] $n\in\hyperlink{LCset}{\mathcal{LC}}$,

\item[$(*)$] \hyperlink{gamma1}{$(\gamma_1)$}, and

\item[$(*)$] \hyperlink{mg}{$(\on{mg})$}.
\end{itemize}
Then Theorem \ref{dual14} can be applied to $N$.
\end{corollary}

\begin{proof}
By \eqref{mgstrangedc} we get that \hyperlink{mg}{$(\on{mg})$} implies \eqref{betweenmgdc}, the other assertions follow immediately.
\end{proof}

Similarly, the above results can be used to construct sequences $L^1$, $L^2$ for which Theorem \ref{equivalencesmallsequthm} applies.\vspace{6pt}

{\itshape Note:}

\begin{itemize}
\item[$(*)$] A sequence $N$ satisfies the assertions listed in Corollary \ref{dual14cor} if and only if $n$ is formally a so-called {\itshape strongly regular sequence} in the notion of \cite[Sect. 1.1]{Thilliezdivision}. The sequence $M$ in \cite{Thilliezdivision} is precisely denoting $m$ in the notation used in this work.

\item[$(*)$] Corollary \ref{dual14cor} applies to $N\equiv G^s$ for any $s>1$. On the other hand Theorem \ref{dual14} also applies to the so-called $q$-Gevrey sequences given by $M^q:=(q^{p^2})_{p\in\NN}$ with $q>1$. Each $M^q$ violates \hyperlink{mg}{$(\on{mg})$} but \eqref{betweenmgdc} is satisfied.
\end{itemize}

We also have the following result which shows how \eqref{om1} can be obtained for the dual sequence $D$ (and for $L$). This is crucial when $D$ (resp. $L$) shall belong to a family $\fF$ as considered in Section \ref{boundednesssection}.

\begin{proposition}\label{dual15}
Let $N\in\hyperlink{LCset}{\mathcal{LC}}$ be given, let $D\in\hyperlink{LCset}{\mathcal{LC}}$ denote the corresponding dual sequence and let $L$ given by \eqref{almoststuffequ} w.r.t. $D$. We assume that $N$ is also having
\begin{itemize}
\item[$(*)$] $\alpha(\nu)<+\infty$.
\end{itemize}
Then $\beta(\delta)=\beta(\lambda)>0$ and both $\omega_D$ and $\omega_L$ satisfy \eqref{om1}.
\end{proposition}

\begin{proof}
First, by \cite[Thm. 3.16, Cor. 3.17]{index} we know that $\alpha(\nu)<+\infty$ implies (in fact it is even equivalent to) \hyperlink{mg}{$(\on{mg})$}. Consequently, also \eqref{betweenmgdc} holds true, see \eqref{mgstrangedc}. Second, using these facts Theorem \ref{Javithm2143} implies that $\beta(\delta)>0$. Then, by \cite[Thm. 3.11 $(vii)\Leftrightarrow(viii)$]{index} (applied to $\beta=0$) we get $\gamma(D)>0$ as well (for the definition and the study of this growth index $\gamma(\cdot)$ for weight sequences we refer to \cite[Sect. 3.1]{index}). By combining \cite[Cor. 4.6 $(i)$]{index} and \cite[Cor. 2.14]{index} (applied to $\sigma:=\omega_D$) we have that $\omega_D$ satisfies \eqref{om1} and this condition is abbreviated by $(\omega_1)$ in \cite{index}. Finally, the equivalence between $D$ and $L$ clearly preserves \eqref{om1} for $\omega_L$ by definition of the associated weight functions and the equivalence \cite[Thm. 3.1 $(ii)\Leftrightarrow(iii)$]{subaddlike} applied to the sequence $D$.
\end{proof}

Let us combine now Theorem \ref{dual14} and Proposition \ref{dual15}:

\begin{theorem}\label{dual16}
Let $N\in\hyperlink{LCset}{\mathcal{LC}}$ be given, let $D\in\hyperlink{LCset}{\mathcal{LC}}$ denote the corresponding dual sequence and let $L$ be given by \eqref{almoststuffequ} w.r.t. $D$. We assume that $N$ also satisfies
\begin{itemize}
\item[$(*)$] $1<\beta(\nu)\le\alpha(\nu)<+\infty$.
\end{itemize}

Then $L$ is a sequence such that $\lim_{p\rightarrow+\infty}(l_p)^{1/p}= 0$, $(ii)$ and $(iv)$ in Section \ref{answersection} and all requirements from $(i)$ there except $L_0\le L_1$. However, in view of $(iv)$ in Remark \ref{dualremark} also $(i)$ from Section \ref{answersection} can be obtained when passing to $\widetilde{L}$.
\end{theorem}

{\itshape Note:} By applying the technical Proposition \ref{prop:uniformbound} it is possible that, when given a one-parameter family of sequences $N^{(\beta)}$, $\beta>0$, and having the requirements from Theorem \ref{dual16}, to construct from the corresponding family $\mathcal{L}:=\{L^{(\beta)}: \beta>0\}$ (resp. $\widetilde{\mathcal{L}}:=\{\widetilde{L}^{(\beta)}: \beta>0\}$) a technical uniform bound $\mathbf{a}$ as required in Section \ref{boundednesssection} and hence to apply Theorem \ref{finalmarkinthm} to $\mathcal{L}$ (resp. to $\widetilde{\mathcal{L}}$).

\subsection{The bidual sequence}
The goal of this final Section is to show how the procedure from Section \ref{maindualsection} can be reversed in a canonical way. Let us first recall: For any $N\in\hyperlink{LCset}{\mathcal{LC}}$ we have that the corresponding dual sequence $D\in\hyperlink{LCset}{\mathcal{LC}}$ and so in \cite[Definition 2.1.41, p. 81]{dissertationjimenez} the following natural definition has been given:
\begin{equation}\label{bidualdef}
\forall\;p\ge\delta_1=1:\;\;\;\epsilon_{p+1}:=\Sigma_D(p),\hspace{20pt}\epsilon_0=\epsilon_1:=1,
\end{equation}
and set $E_p:=\prod_{i=1}^p\epsilon_i$. Finally we put $E_0:=1$ and so $E\in\hyperlink{LCset}{\mathcal{LC}}$ with $1=E_0=E_1$ follows by definition. This sequence $E=(E_p)_{p\in\NN}$ is called the {\itshape bidual sequence} of $N$ and in \cite[Theorem 2.1.42, p. 81]{dissertationjimenez} it has been proven that $N$ and $E$ are equivalent. (In fact there even a slightly stronger equivalence on the level of the corresponding quotient sequences has been established.)

We prove now converse versions of Lemmas \ref{dual1} and \ref{dual2}.

\begin{lemma}\label{dual4}
Let $D\in\hyperlink{LCset}{\mathcal{LC}}$ be given with $\alpha(\delta)<1$.

Then the (bi)-dual sequence $E$ defined via \eqref{bidualdef} has \eqref{dual1equ1} for some $\beta>1$ (and so $E$ is strongly non-quasianalytic).
\end{lemma}

\begin{proof}
Since $\alpha(\delta)<1$ we have that $D$ satisfies \eqref{betweenmgdc}, see the proof of Proposition \ref{dual15}. Thus $\beta(\epsilon)>1$ follows by Theorem \ref{Javithm2143} and so, for some $\beta>1$, we have
$$\exists\;H\ge 1\;\forall\;1\le p\le q:\;\;\;\frac{\epsilon_p}{p^{\beta}}\le H\frac{\epsilon_q}{q^{\beta}},$$
i.e. $p\mapsto\frac{\epsilon_p}{p^{\beta}}$ is almost increasing.
\end{proof}

\begin{lemma}\label{dual5}
Let $D\in\hyperlink{LCset}{\mathcal{LC}}$ be given with $\lim_{p\rightarrow+\infty}\delta_p/p=0$ (resp. equivalently $\lim_{p\rightarrow+\infty}(d_p)^{1/p}=0$). Then the dual sequence $E$ satisfies $\lim_{p\rightarrow+\infty}\epsilon_p/p=\lim_{p\rightarrow+\infty}(e_p)^{1/p}=+\infty$.
\end{lemma}

\begin{proof}
First, $\lim_{p\rightarrow+\infty}\delta_p/p=0$ if and only if $\lim_{p\rightarrow+\infty}(d_p)^{1/p}=0$ holds by \eqref{mucompare}.

Let $C\ge 1$ be given, arbitrary but from now on fixed and w.l.o.g. we can take $C\in\NN_{>0}$. Then we find some $p_C\in\NN_{>0}$ such that $\delta_p\le pC^{-1}$ for all $p\ge p_C$ holds true. For all such (large) integers $p$ we also have $pC\ge p_C$ and so $\delta_{pC}\le(pC)C^{-1}=p$ for all $p\ge p_C$. By definition, since $\epsilon_{p+1}=\Sigma_D(p)=|\{j\in\NN_{>0}: \delta_j\le p\}|$ and $j\mapsto\delta_j$ is non-decreasing, we get now $\epsilon_{p+1}\ge pC\Leftrightarrow\frac{\epsilon_{p+1}}{p}\ge C$ for all $p\ge p_C$. Thus we are done because $C$ is arbitrary (large).
\end{proof}

Finally we get the following main result.

\begin{theorem}\label{dualdual}
Let $D\in\hyperlink{LCset}{\mathcal{LC}}$ be given with $1=D_0=D_1$ and assume that

\begin{itemize}
\item[$(*)$] $\alpha(\delta)<1$.
\end{itemize}

Then one can apply Theorem \ref{thm:main} to the sequence $L$ given by \eqref{almoststuffequ} and, in addition, the isomorphisms from Theorem \ref{thm:main} hold for the classes defined via $D$ too (by Remark \ref{mainthmremark}). The corresponding dual sequence $E\in\hyperlink{LCset}{\mathcal{LC}}$ (see \eqref{bidualdef}) is strong non-quasianalytic.
\end{theorem}

\begin{proof}
The first part holds by comment $(b)$ in Section \ref{Matuszeskasection} applied to $D$ (even under more general assumptions on the given sequence). The strong non-quasianalyticity for $E$ follows from Lemma \ref{dual4}.
\end{proof}


\end{document}